\documentclass[11pt]{article}

\usepackage[T1]{fontenc}
\usepackage[utf8]{inputenc}
\usepackage{lmodern}
\usepackage{amsmath,amssymb,amsthm,mathtools}
\usepackage[a4paper,margin=29mm]{geometry}
\usepackage{microtype}
\usepackage{enumitem}
\usepackage[colorlinks=true,linkcolor=blue,citecolor=blue,urlcolor=blue]{hyperref}

\newtheorem{theorem}{Theorem}[section]
\newtheorem{proposition}[theorem]{Proposition}
\newtheorem{corollary}[theorem]{Corollary}
\newtheorem{lemma}[theorem]{Lemma}
\theoremstyle{definition}

\theoremstyle{remark}
\newtheorem{remark}[theorem]{Remark}

\newcommand{\Will}{\mathcal W}
\newcommand{\El}{\mathcal E}
\newcommand{\R}{\mathbb R}
\newcommand{\Sph}{\mathbb S}
\newcommand{\dd}{\,\mathrm d}
\newcommand{\Area}{\operatorname{Area}}
\newcommand{\Diff}{\operatorname{Diff}}
\newcommand{\id}{\operatorname{id}}
\newcommand{\BL}{\mathrm{BL}}
\newcommand{\TV}{\mathrm{TV}}
\newcommand{\Var}{\mathbf V}
\newcommand{\Gr}{\mathrm G}
\newcommand{\Span}{\operatorname{span}}
\newcommand{\vol}{\operatorname{vol}}
\newcommand{\moduli}{\mathfrak M_1}
\newcommand{\Gcal}{\mathcal G}
\newcommand{\Dcal}{\mathcal D}

\numberwithin{equation}{section}

\title{Hopf Reduction and Multiplicity-Resolved Endpoints\\
for the Willmore Flow of Hopf Tori}
\author{
Mohameden Ahmedou\thanks{Mathematisches Institut, Justus-Liebig-Universit\"at
Gie\ss en, Arndtstrasse 2, 35392 Gie\ss en, Germany.
E-mail: \texttt{Mohameden.Ahmedou@math.uni-giessen.de}.}
\and
Ruben Jakob\thanks{Mathematics Department, Technion--Israel Institute of
Technology, 3200003 Haifa, Israel.
E-mail: \texttt{rubenj@technion.ac.il}.}}
\date{\today}

\hypersetup{
 pdfauthor={Mohameden Ahmedou and Ruben Jakob},
 pdftitle={Hopf Reduction and Multiplicity-Resolved Endpoints for the Willmore Flow of Hopf Tori}
}

\begin{document}

\maketitle

\begin{abstract}
We prove smooth convergence modulo reparametrization at arbitrary initial
energy for the classical Willmore flow of Hopf tori in the round
three-sphere. For simple initial immersions and for every finite cover of
the Hopf torus over a parametrized profile, we identify the unique limiting
integral varifold, its multiplicity and the limiting marked period lattice
of the domain torus. We construct simple initial immersions with energies arbitrarily close to
$4\pi^2$ from above whose flows converge exponentially, modulo
reparametrization, to a twice-covered Clifford torus. A \L ojasiewicz--Simon gradient inequality bounds the $L^2$-length of the surface trajectory after
time $t$ by a power of the energy gap for all sufficiently large $t$.
In particular, the entire trajectory has finite $L^2$-length.
Both simple and multiply covered equators have optimal exponent $1/2$.
Their orbits under rotations and reparametrizations attract all nearby
profiles precisely for multiplicities one, two and three. Near the Clifford
torus two embedded Hopf tori can have the same marked conformal class
but noncollinear initial conformal velocities.
\end{abstract}

\medskip
\noindent\textbf{Keywords.}
Willmore flow; Hopf torus; elastic flow; full convergence;
\L ojasiewicz--Simon inequality; varifold multiplicity; conformal modulus.

\smallskip
\noindent\textbf{2020 Mathematics Subject Classification.}
Primary 53E40; Secondary 49Q20, 53C42, 58J35.

\section{Introduction}

A simple Hopf torus in the round three-sphere can converge under the
classical Willmore flow to a Clifford torus covered twice
\cite[Theorems~2.1 and~2.2]{JakobMultiplicity}. The standard Hopf tori
over a once-traversed and a twice-traversed equator have the same
square conformal class. To recover the multiplicity of a general limit,
one must keep track of both the degree of the initial covering
of the Hopf bundle over the parametrized profile and the number of times
the limiting profile visits the base point of a given Hopf fibre. We prove
smooth convergence on the original torus modulo reparametrization and
identify the product of these integers as the limiting varifold density
almost everywhere.

Let $\Sigma$ be a smooth, compact, connected torus and let
$F\colon\Sigma\to\Sph^3$ be a smooth immersion into the unit round
three-sphere.  We use the normalization
\begin{equation}
 \label{eq:W-definition}
 \Will(F):=\int_\Sigma\left(1+\frac14|H_F|^2\right)\dd\mu_F.
\end{equation}
Here $g_F:=F^*g_{\Sph^3}$ is the induced metric, $\mu_F$ is its area
measure, and $H_F$ is the trace mean-curvature vector in $\Sph^3$.
We denote by $A_F$ the second fundamental form in $\Sph^3$ and by
$A_F^\circ:=A_F-\frac12g_F\otimes H_F$ its trace-free part.
For a family $F_t$ we write $\mu_t:=\mu_{F_t}$.
For a field $V$ along $F$ tangent to $\Sph^3$, write
$V=V^\top+V^\perp$ for its orthogonal decomposition along
$DF(T\Sigma)$ and its complement in $T\Sph^3$.
The normal Laplacian is $\Delta^\perp=\operatorname{tr}_{g_F}(\nabla^\perp)^2$.
We write the classical Willmore flow as
\begin{equation}
 \label{eq:WF}
 \partial_tF_t=-G_{F_t},
 \qquad G_F:=\nabla_{L^2}\Will(F).
\end{equation}
Reparametrization invariance of $\Will$ makes $G_F$ perpendicular to
$DF(T\Sigma)$, so \eqref{eq:WF} prescribes the full velocity and gives
$(\partial_tF_t)^\top=0$. A time-dependent reparametrization can add a
tangential velocity. The resulting family satisfies the projected equation
$(\partial_tF_t)^\perp=-G_{F_t}$. For any smooth solution of that equation
the first variation gives
\begin{equation}
 \label{eq:W-dissipation}
 -\frac{\dd}{\dd t}\Will(F_t)
 =-\int_\Sigma\langle G_{F_t},\partial_tF_t\rangle\dd\mu_t
 =\int_\Sigma|G_{F_t}|^2\dd\mu_t
 =\int_\Sigma|(\partial_tF_t)^\perp|^2\dd\mu_t .
\end{equation}
For a solution of \eqref{eq:WF}, the last integral equals
$\|\partial_tF_t\|_{L^2(\mu_t)}^2$. Its $L^2$-length on a time interval
$I$ is $\int_I\|\partial_tF_t\|_{L^2(\mu_t)}\dd t$. In any
time-dependent parametrization the same length is obtained by integrating
the $L^2$-norm of the perpendicular component of the velocity, since this
component and the area measure both pull back under reparametrization.

Let $\pi\colon\Sph^3\to\Sph^2$ be the Hopf map. A smooth regular closed
curve $\gamma$ in the unit two-sphere defines a torus by pulling back the
Hopf circle bundle. Its natural immersion in $\Sph^3$ sends each fibre
onto the Hopf circle over the corresponding profile point. This
parametrized construction retains the different branches at a
self-intersection of $\gamma$. Pinkall's description of Hopf tori
\cite{Pinkall1985} relates the induced geometry to that of the profile.
For a cover of degree $d$ of this abstract torus, the energy and gradient
identities in our normalization are
\begin{equation}
 \label{eq:dictionary-intro}
 \Will(F)=d\pi\El(\gamma),
 \qquad
 D\pi_F(G_F)=4\nabla_{L^2}\El(\gamma),
 \qquad
 \El(\gamma):=\int_{\Sph^1}(1+|\kappa_\gamma|^2)\dd s_{\gamma}.
\end{equation}
Here $\kappa_\gamma$ is the geodesic curvature vector and $\dd s_\gamma$
is spherical arclength. The gradient identity is evaluated at the
corresponding profile point, as in Proposition~\ref{prop:dictionary}.
For a profile $\gamma$, the superscript $\perp$ denotes projection in
$T\Sph^2$ onto the line perpendicular to the curve tangent. In standard
elastic time $\tau=4t$, the equation
$\partial_\tau\gamma=-\nabla_{L^2}\El(\gamma)$ prescribes the full
velocity. Its projected form prescribes only $(\partial_\tau\gamma)^\perp$
and allows a time-dependent change of parameter.

Jakob proved global existence and smooth subconvergence for suitably
parametrized Hopf tori \cite{JakobHopf}. We obtain full convergence on
every finite cover at arbitrary initial energy. On the profile we use the
smooth compactness estimates for spherical elastic flow
\cite{DallAcquaEtAl2018}. A local \L ojasiewicz--Simon inequality and the
estimate of Rupp and Spener for constant-speed parametrizations
\cite{RuppSpener2024} control the distance travelled after the flow enters
a neighborhood of a critical profile. The first-exit argument used on
the sphere in \cite{JakobHopf} then gives full convergence in the
constant-speed parametrization obtained by fixing $0$ on the parameter
circle.
The same gradient inequality bounds the $L^2$-length of the surface
trajectory from time $t$ onward and hence the variation of its area measure.

Pinkall's lattice formula raises a second question. In the standard marking, a real lift
of the oriented area and the length of a profile determine the conformal
parameter of the associated Hopf torus. Do they also determine the direction in which that class initially moves? The lack of ambient M\"obius invariance of the classical Willmore flow leaves this question open even within 
the family of Hopf tori. We construct two embedded profiles with exactly the same area and length. Along the first flow the area is stationary to first order and the length decreases, whereas along the second the area has nonzero
initial derivative. Thus the two flows leave their common point in
Teichm\"uller space in noncollinear directions, even allowing a separate
positive change of time for each trajectory
(Proposition~\ref{prop:nonprojection}).

For tori of revolution in $\R^3$, related results include the $8\pi$
convergence theorem of \cite{DallAcquaMullerSchaetzleSpener2024} and the
analysis of singularities and of flow at fixed conformal class in
\cite{DallAcquaMullerRuppSchlierf2025}. An account of the analysis of
Willmore surfaces is given in \cite{LanMartinoRiviereSurvey}.

We use throughout the a.e.-multiplicity notion of simplicity from
\cite[Definition~1.1]{JakobMultiplicity}, also used in
\cite[Definition~4.3]{JakobMIWF}: a smooth immersion
$F\colon\Sigma\to\Sph^3$ is \emph{simple} if
\begin{equation}
 \label{eq:ae-simple}
 \#F^{-1}(z)=1
 \quad\text{for $\mathcal H^2$-almost every }z\in F(\Sigma).
\end{equation}
The profile need not be embedded: transverse self-intersections are compatible with
\eqref{eq:ae-simple}.  Proposition~\ref{prop:normalization} turns this
a.e.-multiplicity hypothesis into a global one-sheeted parametrized Hopf model.
The main results do not require simplicity: the degree $d$ of the
normalizing cover is retained throughout. The scalar oriented area is denoted
by $\mathcal A$, and $A_F$ continues to denote the second fundamental form.

\begin{theorem}[Parametrized reduction and all-energy convergence]
\label{thm:main}
Let $F_0\colon\Sigma\to\Sph^3$ be a smooth immersion with
$\pi\circ F_0=\gamma_0\circ u_0$, where $u_0\colon\Sigma\to B$ is a smooth
submersion onto a circle and $\gamma_0\colon B\to\Sph^2$ is a smooth
regular closed curve. The pair $(u_0,\gamma_0)$ is part of the data.
Write $d$ for the degree of the covering
$x\mapsto(u_0(x),F_0(x))$ of the pullback torus $M_{\gamma_0}=\gamma_0^*\Sph^3$ - see \eqref{eq:pullback-model} and  Proposition~\ref{prop:normalization} - and let
$F_t$, $t\in[0,T_{\max})$, be the maximal smooth solution of the
Willmore flow \eqref{eq:WF} with initial datum $F_0$.  Then the
following assertions hold.
\begin{enumerate}[label=\textnormal{(\roman*)},leftmargin=8mm]
 \item The solution exists smoothly for every $t\geq0$.  
 There are smooth regular closed profiles $\gamma_t$, their abstract parametrized Hopf tori $M_{\gamma_t}$, standard Hopf immersions $X_{\gamma_t}\colon M_{\gamma_t}\to \Sph^3$  
 as in \eqref{eq:pullback-model} and $d$-sheeted coverings $Q_t\colon\Sigma\to M_{\gamma_t}$ such that
 \begin{equation}
  \label{eq:parametrized-reduction-main}
  F_t=X_{\gamma_t}\circ Q_t,
  \qquad
  (\partial_t\gamma_t)^\perp=-4\nabla_{L^2}\El(\gamma_t).
 \end{equation}
 Consequently,
 \begin{equation}
  \label{eq:finite-time-multiplicity-main}
  \#F_t^{-1}(z)=d\,\#\gamma_t^{-1}(\pi(z))
  \quad\text{for every }z\in F_t(\Sigma).
 \end{equation}

 \item There are a smooth regular closed critical point
 $\gamma_\infty\colon\Sph^1\to \Sph^2$ of $\mathcal{E}$, 
 a smooth Hopf immersion
 $F_\infty\colon\Sigma\to \Sph^3$ parametrizing the 
 same Hopf torus in $\Sph^3$ as $X_{\gamma_\infty}$, 
 and diffeomorphisms $\Psi_t\in\operatorname{Diff}(\Sigma)$ such that, 
 for every $k\in\mathbb N_0$,
 \begin{equation}
  \label{eq:smooth-main}
  F_t\circ\Psi_t\longrightarrow F_\infty
  \quad\text{in }C^k(\Sigma,\R^4).
 \end{equation}
 In particular $F_\infty$ is Willmore, and there holds:
 \begin{equation}
  \label{eq:limit-energy-main}
  \Will_\infty:=\lim_{t\to\infty}\Will(F_t)
  =\Will(F_\infty)=d\pi\El(\gamma_\infty).
 \end{equation}

\end{enumerate}
\end{theorem}

The next theorem describes the endpoint of the unreparametrized surface flow
and the limiting marked lattice.

For a smooth embedded surface $M\subset\Sph^3$, we write $|M|$ for the
multiplicity-one integral two-varifold induced by the inclusion
$M\hookrightarrow\R^4$.

\begin{theorem}[Multiplicity-resolved endpoint]
\label{thm:endpoint}
Under the hypotheses and with the notation of Theorem~\ref{thm:main}, the
following assertions hold.
\begin{enumerate}[label=\textnormal{(\roman*)},leftmargin=8mm]

 \item The solution $F_t$ of \eqref{eq:WF} has finite $L^2$-length,
 \begin{equation}
  \label{eq:finite-length-main}
  \int_0^\infty
  \|\partial_tF_t\|_{L^2(\mu_t)}\dd t<\infty.
 \end{equation}
 More precisely, there are $t_0\geq0$, $C<\infty$, and a
 \L ojasiewicz--Simon exponent $\theta\in(0,\frac12]$ at
 $\gamma_\infty$ such that
 \begin{equation}
  \label{eq:tail-main}
  L_F(t):=\int_t^\infty
  \|\partial_sF_s\|_{L^2(\mu_s)}\dd s
  \leq C\bigl(\Will(F_t)-\Will_\infty\bigr)^\theta
  \quad(t\geq t_0).
 \end{equation}
 Here $\theta$ is taken in the convention of Lemma~\ref{lem:LS-local}:
 the power of the energy gap in the gradient inequality is $1-\theta$.

 \item On the fixed parameter domain, $\nu_t:=\mu_t$ converges in total
 variation to a Radon measure $\nu_\infty$.  The image measures
 $\lambda_t:=(F_t)_\#\mu_t$ converge in bounded-Lipschitz distance to
 $\lambda_\infty$.  If $\Var_t$ is the integral two-varifold induced by
 $F_t$, then the full trajectory converges as integral varifolds and
 \begin{equation}
  \label{eq:varifold-main}
  \Var_t\rightharpoonup\Var_\infty=\Var_{F_\infty},
  \qquad \|\Var_\infty\|=\lambda_\infty.
 \end{equation}
 Quantitatively, with $W_0:=\Will(F_0)$,
 \begin{equation}
  \label{eq:measure-rate-main}
  \|\nu_\infty-\nu_t\|_{\TV}
  +d_{\BL}(\lambda_\infty,\lambda_t)
  \leq4\sqrt{W_0}\,L_F(t).
 \end{equation}

 \item For $\mathcal H^2$-almost every
 $z\in F_\infty(\Sigma)$,
 \begin{equation}
  \label{eq:multiplicity-main}
  \Theta^2(\|\Var_\infty\|,z)
  =d\,\#\gamma_\infty^{-1}(\pi(z)).
 \end{equation}
 Consequently, if $\gamma_\infty$ is an $m$-fold cover of an embedded
 closed curve $\bar\gamma_\infty$, then
 \begin{equation}
  \label{eq:mcover-main}
  \Var_\infty=dm\,|\pi^{-1}(\bar\gamma_\infty(\Sph^1))|.
 \end{equation}

 \item Let $L(t):=L(\gamma_t)$ and let $[\mathcal A(t)]\in
 \R/(4\pi\mathbb Z)$ be its oriented area class. Choose a real representative
 $\mathcal A_\infty$ at the limit and the final-time continuous lift furnished
 by Lemma~\ref{lem:anchored-lift}. Then
 \begin{equation}
 \label{eq:pinkall-main}
 L(t)\to L_\infty>0,\qquad \mathcal A(t)\to\mathcal A_\infty,
 \qquad \Gamma_t\to\Gamma_\infty,
 \end{equation}
 where the lattice of the abstract profile torus is
 \begin{equation}
 \label{eq:lattice-main}
 \Gamma_t=\Span_{\mathbb Z}\{(2\pi,0),(\mathcal A(t)/2,L(t)/2)\}.
 \end{equation}
 Its marked generator matrices satisfy
 \begin{equation}
 \label{eq:marked-generator-convergence}
 B(t)=\begin{pmatrix}2\pi&\mathcal A(t)/2\\0&L(t)/2\end{pmatrix}
 \longrightarrow B_\infty=
 \begin{pmatrix}2\pi&\mathcal A_\infty/2\\0&L_\infty/2\end{pmatrix}.
 \end{equation}
 After a fixed choice of marking, the actual covered surface has lattice
 $B(t)M\mathbb Z^2$, where
 $M=\left(\begin{smallmatrix}a&c\\b&e\end{smallmatrix}\right)$ is an
 integer matrix with $\det M=d>0$. With $z(t)=(\mathcal A(t)+iL(t))/(4\pi)$,
 its marked modulus and conformal class satisfy
 \begin{equation}
 \label{eq:modulus-main}
 z_M(t)=\frac{c+ez(t)}{a+bz(t)}\longrightarrow
 \frac{c+ez_\infty}{a+bz_\infty},\qquad
 [z_M(t)]\longrightarrow[z_{M,\infty}]
 \quad\text{in }\moduli.
 \end{equation}
 Here $\moduli=\mathrm{PSL}(2,\mathbb Z)\backslash\mathbb H_+$ and
 $\mathbb H_+=\{z\in\mathbb C:\operatorname{Im}z>0\}$.
 For $d=1$ one may choose $M=I$.

\end{enumerate}

Finally, if $e(t):=\Will(F_t)-\Will_\infty$, then, after increasing $t_0$,
\begin{equation}
 \label{eq:energy-rates-main}
 e(t)\leq
 \begin{cases}
  Ce^{-ct},&\theta=\frac12,\\[2mm]
  C(1+t)^{-1/(1-2\theta)},&0<\theta<\frac12.
 \end{cases}
\end{equation}
All constants may depend on the flowline and its limiting critical profile
from Theorem~\ref{thm:main}(ii).
\end{theorem}

For the once- and twice-traversed equators in the opening example, the
lattice covolumes are $2\pi^2$ and $4\pi^2$ and the varifold densities
are one and two, respectively. Other degree-two covers of the Clifford
torus can have a different conformal class
(Section~\ref{sec:double-Clifford}).

The following refinements distinguish convergence to a covered equator from
attraction of all sufficiently nearby initial profiles. For $m\geq1$, write
\[
 \eta_m(s)=(\cos s,\sin s,0),\qquad s\in\R/(2\pi m\mathbb Z),
 \qquad \ell_m=2\pi m.
\]
A normal graph over a profile $\eta$ with unit normal $n$ in $\Sph^2$
is the curve $s\mapsto\exp_{\eta(s)}(u(s)n(s))$. Over $\eta_m$ its
height $u$ is a function on the parameter circle of length $\ell_m$,
and its graph norm is $\|u\|_{H^2}$ on that circle. A neighborhood
modulo rotations and reparametrizations consists of profiles admitting
such a graph representation, with small norm, over a rotation of $\eta_m$.

\begin{theorem}[Optimal exponent and rates at covered equators]
\label{thm:equator-rates}
For every integer $m\geq1$ there are $r_m,C_m>0$ such that every smooth
normal graph $\gamma$ with $H^2$ graph norm less than $r_m$ over a rotation
of $\eta_m$ satisfies
\begin{equation}
 \label{eq:optimal-LS-main}
 |\El(\gamma)-\ell_m|^{1/2}
 \leq C_m\|\nabla_{L^2}\El(\gamma)\|_{L^2(\dd s_\gamma)}.
\end{equation}
The only critical curves in this neighborhood are rotated copies of
$\eta_m$, up to reparametrization. If the limiting profile in
Theorem~\ref{thm:main} is such a curve, its \L ojasiewicz--Simon exponent
can be chosen to be $1/2$. The energy gap, the length $L_F(t)$ in
\eqref{eq:tail-main}, the measure
errors in \eqref{eq:measure-rate-main}, and the hyperbolic distance between the marked moduli
$z_M(t)$ and $z_{M,\infty}$ decay exponentially.
The constant-speed profile of Proposition~\ref{prop:profile-full}
converges exponentially in every $C^k$. The diffeomorphisms $\Psi_t$ in
\eqref{eq:smooth-main} can be chosen so that the reparametrized surfaces
do as well. The constants and rates may depend on $k$ and the trajectory. In addition, for all sufficiently large $t$,
\begin{equation}
 \label{eq:quadratic-length-main}
 |L(\gamma_t)-\ell_m|\leq C\bigl(\El(\gamma_t)-\ell_m\bigr).
\end{equation}
\end{theorem}

\begin{theorem}[Attracting covered equators]
\label{thm:equator-stability}
For $m=1,2,3$ there is an $H^2$ neighborhood of the rotation and
reparametrization orbit of $\eta_m$ such that the elastic profile flow
from every smooth profile in this neighborhood converges exponentially,
modulo reparametrization, to a rotated $m$-covered equator. Every fixed
$d$-sheeted Hopf cover consequently converges, modulo domain
reparametrization, to a Clifford immersion of total multiplicity $dm$, 
with the convergence rates of Theorem~\ref{thm:equator-rates}.
For every integer $m\geq1$, the Hessian of $\El$ on normal graphs 
has kernel exactly the two rotation modes, and its Morse index is
\begin{equation}
 \label{eq:Morse-index-main}
 2\bigl(m-1-\lfloor m/\sqrt2\rfloor\bigr).
\end{equation}
In particular, for every $m\geq4$ there are arbitrarily small smooth
perturbations whose flows do not converge to this rotation and
reparametrization orbit. Thus the list $m=1,2,3$ for an attracting
neighborhood is sharp.
\end{theorem}

For the unit sphere and length coefficient one, Langer and Singer's
variational stability criterion gives precisely $m=1,2,3$
\cite[Theorem~3.1]{LangerSinger1987}. Theorem~\ref{thm:equator-stability}
also gives the index and establishes local attraction for the $L^2$ flow.

For $m=2$ we apply this attraction result to an explicit perturbation
of the twice-traversed equator. Its single transverse double point gives
a simple initial Hopf immersion, and its flow converges exponentially to
a twice-covered Clifford torus (Corollary~\ref{cor:double-Clifford}).

Proposition~\ref{prop:classical-estimates} bounds the squared hyperbolic
speed of the marked conformal class by the elastic dissipation.
Proposition~\ref{prop:flow-comparison} compares the classical and
M\"obius-invariant Willmore flows on the profiles themselves. Their linearizations at a covered equator
differ by a factor of four. At nearby profiles a single change of time
can already fail to identify their normal velocities.

\section{The exact Hopf--elastic dictionary}
\label{sec:dictionary}

We identify $\Sph^3$ with the unit quaternions and $\Sph^2$ with the unit
sphere in $\operatorname{Im}\mathbb H$, oriented as the boundary of the
space with ordered basis $(i,j,k)$.  We use the Hopf map
\begin{equation}
 \label{eq:Hopf-map}
 \pi(q)=q^{-1}iq\in\Sph^2\subset\operatorname{Im}\mathbb H.
\end{equation}
To relate this convention to Pinkall's, write
$q=a+ib+jc+kd$ and let $\overline q=a-ib-jc-kd$ be quaternionic
conjugation.  Pinkall's involution and Hopf map are
\[
 \widetilde q=a-ib+jc+kd=-i\overline q\,i,
 \qquad
 \pi_P(q)=\widetilde q\,q
 \in\Sph^3\cap\Span_{\R}\{1,j,k\};
\]
see \cite[p.~380, formula~(1)]{Pinkall1985} and
\cite[formula~(75)]{JakobMIWF}.  Since $q^{-1}=\overline q$ on $\Sph^3$,
\[
 \pi_P(q)=-i\pi(q),\qquad \pi=L_i\circ\pi_P,
 \qquad L_i(y):=iy.
\]
The map $L_i$ is a linear isometry from $\Span_{\R}\{1,j,k\}$ onto
$\operatorname{Im}\mathbb H$ and preserves the orientations determined by
the ordered bases $(1,j,k)$ and $(i,j,k)$.  Thus these two formulas describe
the same circle fibres with isometrically identified base spheres.
In particular,
\[
 \pi(e^{i\varphi}q)=\pi(q),\qquad
 \pi^{-1}(\pi(q))=\{e^{i\varphi}q:\varphi\in\R\}.
\]
The second identity follows because a unit quaternion commutes with $i$
exactly when it belongs to $\{e^{i\varphi}:\varphi\in\R\}$.
The vertical line at $q$ is $\Span_{\R}\{iq\}$ and its orthogonal complement
in $T_q\Sph^3$ is the horizontal space
$\mathcal H_q=\{uq:u\in\Span_{\R}\{j,k\}\}$.
For $u\in\operatorname{Im}\mathbb H$, differentiation gives
\[
 D\pi_q(uq)=q^{-1}(iu-ui)q.
\]
For horizontal $uq$, one has $iu-ui=2iu$, so $D\pi_q$ restricts to a
homothety of factor $2$ on $\mathcal H_q$.

Throughout, smooth means $C^\infty$, and a closed curve is parametrized by
a smooth circle.  Such a curve is regular if its differential is nonzero at
every point of its parameter circle.

We first fix the intrinsic parametrized model.  If
$\gamma\colon B\to\Sph^2$ is a smooth regular closed curve with parameter
circle $B$, set
\begin{equation}
 \label{eq:pullback-model}
 M_\gamma:=\gamma^*\Sph^3
 :=\{(b,q)\in B\times\Sph^3:\gamma(b)=\pi(q)\},
 \qquad
 p_\gamma(b,q):=b,
 \qquad
 X_\gamma(b,q):=q.
\end{equation}
The projection $p_\gamma$ makes $M_\gamma$ a principal circle bundle over
$B$.  Every principal circle bundle over a circle is trivial because
$H^2(B;\mathbb Z)=0$; hence $M_\gamma$ is a smooth, compact, connected torus.
The map
$X_\gamma\colon M_\gamma\to\Sph^3$ is an immersion: if
$(v,w)\in T_{(b,q)}M_\gamma$ and $w=0$, then
$D\gamma_b(v)=D\pi_q(w)=0$, whence $v=0$ because $\gamma$ is regular.

We next record the marked flat presentation of this intrinsic model.  Choose
an orientation and a base point on $B$, and let
$a_\gamma\colon\R/\ell\mathbb Z\to B$ be the corresponding
orientation-preserving parametrization for which
$\gamma\circ a_\gamma$ has constant speed $2$.  Thus
$L(\gamma)=2\ell$.  If $\eta\colon\R\to\Sph^3$ is an anchored horizontal
lift of $\gamma\circ a_\gamma$, there is a unique holonomy class
$[h_\gamma]\in\R/(2\pi\mathbb Z)$ such that
\begin{equation}
 \label{eq:horizontal-holonomy}
 \eta(s+\ell)=e^{-ih_\gamma}\eta(s).
\end{equation}
Choose any real representative $h_\gamma$ and set
\begin{equation}
 \label{eq:parametrized-lattice}
 \Gamma_\gamma:=
 \Span_{\mathbb Z}\{(2\pi,0),(h_\gamma,\ell)\}.
\end{equation}
The formula
\begin{equation}
 \label{eq:lattice-pullback-identification}
 I_\gamma\colon\R^2/\Gamma_\gamma\longrightarrow M_\gamma,
 \qquad
 I_\gamma[\varphi,s]
 :=\bigl(a_\gamma([s]),e^{i\varphi}\eta(s)\bigr)
\end{equation}
defines a diffeomorphism.  Indeed, the two generators in
\eqref{eq:parametrized-lattice} act trivially in
\eqref{eq:lattice-pullback-identification}, and the map is a fibrewise
equivariant bijection covering $a_\gamma$.  In these coordinates,
\begin{equation}
 \label{eq:standard-Hopf-model}
 (X_\gamma\circ I_\gamma)[\varphi,s]=e^{i\varphi}\eta(s),
 \qquad
 (p_\gamma\circ I_\gamma)[\varphi,s]=a_\gamma([s]).
\end{equation}
Thus $M_\gamma$ is intrinsic, while
$\R^2/\Gamma_\gamma$ is its marked flat presentation.  We orient
$M_\gamma$ by the ordered coordinates $(\varphi,s)$.  The covering in Proposition~\ref{prop:normalization} induces the orientation
on $\Sigma$ used for its conformal class.
With the sign in \eqref{eq:horizontal-holonomy}, the holonomy computation
in \eqref{eq:holonomy-current} gives
\begin{equation}
 \label{eq:holonomy-area}
 [h_\gamma]=[\mathcal A(\gamma)/2]\in\R/(2\pi\mathbb Z),
\end{equation}
where $[\mathcal A(\gamma)]\in\R/(4\pi\mathbb Z)$ is the oriented spherical
area of the parametrized cycle.  Thus \eqref{eq:parametrized-lattice} is
Pinkall's lattice \cite[Proposition~1]{Pinkall1985}, including all cases in
which the curve $\gamma$ is immersed but not embedded, in particular cases
in which $\gamma$ traverses its trace multiple times.

For clarity, a \emph{parametrized Hopf immersion} is a smooth immersion
$F\colon\Sigma\to\Sph^3$ of a smooth, compact, connected torus for which
there are a smooth circle $B$, a smooth submersion
$u\colon\Sigma\to B$, and a smooth regular closed
curve $\gamma\colon B\to\Sph^2$ such that
\begin{equation}
 \label{eq:Hopf-factorization}
 \pi\circ F=\gamma\circ u.
\end{equation}
The circle $B$ and the maps $u,\gamma$ are part of the parametrized data.
Fibres of $u$ may have several components. This convention includes every
connected finite cover of a standard Hopf model.
The smoothness of $\gamma$ also follows from the factorization: the
submersion $u$ is onto by compactness and connectedness, and on every
smooth local section $\zeta$ over $U$ one has $\gamma|_U=\pi\circ F\circ\zeta$.
Proposition~\ref{prop:normalization} shows that
$F(\Sigma)=\pi^{-1}(\gamma(B))$.  The standard Hopf models and the initial
immersions constructed in \cite{JakobHopf,JakobMultiplicity} have this
factorization.

\begin{proposition}[Parametrized Hopf normalization]
\label{prop:normalization}
Let $F\colon\Sigma\to\Sph^3$ be a parametrized Hopf immersion, with data
$(B,u,\gamma)$ as in \eqref{eq:Hopf-factorization}.  Then there is an integer
$m_F\geq1$ such that the explicitly defined map
\begin{equation}
 \label{eq:normalizing-cover}
 q_F\colon\Sigma\longrightarrow M_\gamma,
 \qquad q_F(x):=(u(x),F(x)),
 \qquad F=X_\gamma\circ q_F,
\end{equation}
is a smooth $m_F$-sheeted covering.  Moreover, for every $z\in F(\Sigma)$,
\begin{equation}
 \label{eq:normalization-count}
 \#F^{-1}(z)=m_F\,\#\gamma^{-1}(\pi(z)).
\end{equation}
If $F$ is simple in the a.e. sense \eqref{eq:ae-simple}, then
$m_F=1$, the profile is one-to-one almost everywhere on its trace, and
$q_F$ is a diffeomorphism.
\end{proposition}

\begin{proof}
The factorization \eqref{eq:Hopf-factorization} shows that $q_F$ takes values
in the circle bundle $M_\gamma$ by its definition in \eqref{eq:pullback-model}, and
the projection $X_\gamma$ onto the second factor of $M_\gamma$ gives 
$F=X_\gamma\circ q_F$.  Its differential is
\begin{equation}
 \label{eq:normalizing-map-differential}
 D(q_F)_x(v)=\bigl(Du_x(v),DF_x(v)\bigr).
\end{equation}
If this vector vanishes, then $DF_x(v)=0$, and hence $v=0$ because $F$ is an
immersion.  Since the source and target are two-dimensional,
$D(q_F)_x$ is an isomorphism at every $x$; thus $q_F$ is a local
diffeomorphism.

The target $M_\gamma$ is connected, and the image of $q_F$ is open because
$q_F$ is a local diffeomorphism.  It is also compact, hence closed, because
$\Sigma$ is compact and $M_\gamma$ is Hausdorff.  Therefore $q_F$ is
surjective.  A proper surjective local diffeomorphism is a finite covering;
denote its sheet number by $m_F\geq1$.  In particular, for every $b\in B$ and
$q\in\pi^{-1}(\gamma(b))$, the restriction
\begin{equation}
 F|_{C_b}\colon C_b:=u^{-1}(b)\longrightarrow
 \pi^{-1}(\gamma(b))
 \label{eq:fibre-covering}
\end{equation}
has exactly $m_F$ preimages over $q$.  Here $C_b$ is a finite disjoint union of circles: it is a compact
one-manifold without boundary. The restriction is a local diffeomorphism
because $q_F$ is, and is an $m_F$-sheeted covering in total.

For every $z\in X_\gamma(M_\gamma)=F(\Sigma)$ one has the exact identity
\begin{equation}
 X_\gamma^{-1}(z)
 =\{(b,z):b\in\gamma^{-1}(\pi(z))\}.
 \label{eq:standard-model-fibre}
\end{equation}
Each point on the right has exactly $m_F$ preimages under $q_F$, which proves
\eqref{eq:normalization-count}.  If $F$ is simple, then
\eqref{eq:normalization-count} and the positivity of
$\mathcal H^2(F(\Sigma))$ force $m_F=1$.  Hence $q_F$ is a one-sheeted
covering and therefore a diffeomorphism.  The same identity then yields
$\#\gamma^{-1}(\pi(z))=1$ for $\mathcal H^2$-almost every image point 
$z\in F(\Sigma)$.
Applying the coarea formula to the local Hopf-cylinder charts, or simply
Fubini's theorem in the coordinates
\eqref{eq:lattice-pullback-identification}, gives
$\#\gamma^{-1}(y)=1$ for $\mathcal H^1$-almost every
$y\in\gamma(B)$.
\end{proof}

Let now $F=X_\gamma\circ Q$, where $Q\colon\Sigma\to M_\gamma$ is a
$d$-sheeted covering, $d\geq1$. Proposition~\ref{prop:normalization} puts
every parametrized Hopf immersion in this form. The degree $d$ will be
constant along the flow, but the profile may acquire additional
multiplicity at its limit.

\begin{proposition}[Normalized Hopf dictionary]
\label{prop:dictionary}
For the functionals in \eqref{eq:W-definition} and
\eqref{eq:dictionary-intro},
\begin{align}
 \Will(F)&=d\pi\El(\gamma),
 \label{eq:W=piE}\\
 D\pi_F(G_F)
 &=4\bigl(\nabla_{L^2}\El(\gamma)\bigr)
       \circ p_\gamma\circ Q,
 \label{eq:projected-gradient}\\
 \int_\Sigma|G_F|^2\dd\mu_F
 &=4d\pi\int_B|\nabla_{L^2}\El(\gamma)|^2\dd s_\gamma,
 \label{eq:gradient-norm-dictionary}
\end{align}
where $p_\gamma$ denotes the bundle projection which was defined in
equation \eqref{eq:pullback-model}, and $\dd s_\gamma$ is spherical
arclength measure on the parametrized profile.
Consequently, if a smooth family $F_t=X_{\gamma_t}\circ Q_t$ of these
models satisfies $(\partial_tF_t)^\perp=-G_{F_t}$, the corresponding profiles satisfy
\begin{equation}
 \label{eq:profile-t}
 (\partial_t\gamma_t)^\perp=-4\nabla_{L^2}\El(\gamma_t).
\end{equation}
Thus
\begin{equation}
 \label{eq:standard-elastic-clock}
 \bar\gamma_\tau:=\gamma_{\tau/4}
 \quad\Longrightarrow\quad
 (\partial_\tau\bar\gamma_\tau)^\perp
 =-\nabla_{L^2}\El(\bar\gamma_\tau).
\end{equation}
The perpendicular surface speed satisfies
\begin{equation}
 \label{eq:speed-clock}
 \|(\partial_tF_t)^\perp\|_{L^2(\mu_t)}
 =2\sqrt{d\pi}\,
 \|\nabla_{L^2}\El(\bar\gamma_{4t})\|_{L^2(\dd s_{\bar\gamma_{4t}})}.
\end{equation}
For a solution of \eqref{eq:WF} the left-hand side is the full surface speed.
\end{proposition}

\begin{proof}
Let
\[
 \rho_\gamma\colon\R^2\longrightarrow\R^2/\Gamma_\gamma,
 \qquad
 q_\gamma:=I_\gamma\circ\rho_\gamma\colon\R^2\longrightarrow M_\gamma
\]
be the quotient map followed by the diffeomorphism in
\eqref{eq:lattice-pullback-identification}.  We perform the local
calculation on the universal covering and descend it only afterwards.  By
\eqref{eq:standard-Hopf-model}, the lifted immersion and the lifted bundle
projection are
\begin{equation}
 \label{eq:lifted-Hopf-model}
 \widehat X_\gamma:=X_\gamma\circ q_\gamma,
 \qquad
 \widehat X_\gamma(\varphi,s)=e^{i\varphi}\eta(s),
 \qquad
 p_\gamma\circ q_\gamma(\varphi,s)=a_\gamma([s]), 
\end{equation}
where $\eta\colon\R\to\Sph^3$ is an anchored horizontal
lift of $\gamma\circ a_\gamma\colon\R/\ell\mathbb Z\to \Sph^2$, 
satisfying relation \eqref{eq:horizontal-holonomy} for the chosen real
representative $h_\gamma$ of its holonomy class.
	
For the two generators of $\Gamma_\gamma$ in \eqref{eq:parametrized-lattice} one has
\begin{align*}
 \widehat X_\gamma(\varphi+2\pi,s)
 &=\widehat X_\gamma(\varphi,s),\\
 \widehat X_\gamma(\varphi+h_\gamma,s+\ell)
 &=e^{i(\varphi+h_\gamma)}\eta(s+\ell)
 =e^{i\varphi}\eta(s)
 =\widehat X_\gamma(\varphi,s),
\end{align*}
where the second identity follows from \eqref{eq:horizontal-holonomy}.
The derivatives are periodic under the same translations, so the tensors
computed below descend to $M_\gamma$.

Put $\widetilde\gamma(s):=\gamma(a_\gamma([s]))$.  The curve
$\widetilde\gamma$ has speed $2$, so its spherical arclength parameter is
$r=2s$. With $T_\gamma=\partial_r\gamma$, choose the oriented unit normal
$n_\gamma=\gamma\times T_\gamma$ and write
$\kappa_\gamma=k_\gamma n_\gamma$; in the following covering-space
formulae these fields are evaluated at $a_\gamma([s])$.  The horizontal
dilation of $\pi$ is $2$, whereas the Hopf fibres have unit speed.  Hence
there is a unique choice of unit normal $\widehat N_\gamma$ along
$\widehat X_\gamma$ for which
\begin{equation}
 \label{eq:normal-dilation}
 D\pi_{\widehat X_\gamma}(\widehat N_\gamma)=2n_\gamma.
\end{equation}
Both sides are invariant under $\Gamma_\gamma$, so
$\widehat N_\gamma$ is invariant as well and descends to a unit normal of
$X_\gamma$ on $M_\gamma$.
Moreover,
$(\partial_s\widehat X_\gamma,\partial_\varphi\widehat X_\gamma)$ is
an orthonormal frame. Along $\varphi=0$, write $\eta'=u\eta$ with $u=u(s)$
a unit vector in $\Span_\R\{j,k\}$. Then $u'=b\,iu$ for a real scalar $b$.
The spherical unit tangent and normal are
$T=\eta^{-1}iu\eta$ and $n=-\eta^{-1}u\eta$, so
$\partial_sT=-2\widetilde\gamma+b n$. Since $r=2s$, this gives
$b=2k_\gamma$. The chosen surface normal is $i\eta'$, as
$D\pi_\eta(i\eta')=2n$. Now $\eta''=-\eta+b\,i\eta'$,
$\partial_\varphi\partial_s\widehat X_\gamma=i\eta'$ at $\varphi=0$,
and $\partial_\varphi^2\widehat X_\gamma=-\widehat X_\gamma$.
Fibre rotation transports these identities to every $\varphi$, giving
\begin{equation}
 \label{eq:Hopf-second-fundamental}
 q_\gamma^*A_{X_\gamma}
 =\widehat N_\gamma
 \begin{pmatrix}2k_\gamma&1\\1&0\end{pmatrix},
 \qquad
 H_{X_\gamma}\circ q_\gamma=2k_\gamma\widehat N_\gamma.
\end{equation}

The half-open rectangle
$\mathcal R:=[0,2\pi)\times[0,\ell)$ is a measurable fundamental domain for
the lattice generated by $(2\pi,0)$ and $(h_\gamma,\ell)$.  It is not being
identified with a Cartesian product model of $M_\gamma$; its horizontal
edges are glued with the holonomy shift $h_\gamma$.  
Since the lifted metric is $\dd\varphi^2+\dd s^2$, \eqref{eq:Hopf-second-fundamental} and $\dd r=2\dd s$ give
\begin{align}
 \Will(X_\gamma)
 &=\int_{\mathcal R}(1+k_\gamma^2)\dd\varphi\dd s \notag\\
 &=2\pi\int_0^\ell(1+k_\gamma^2)\dd s
 =\pi\int_0^{L(\gamma)}(1+|\kappa_\gamma|^2)\dd r
 =\pi\El(\gamma).
 \label{eq:energy-descent}
\end{align}

In spherical arclength, the $L^2$-gradient of the elastic energy is
\begin{equation}
 \nabla_{L^2}\El(\gamma)
 =\bigl(2\partial_r^2k_\gamma+k_\gamma^3+k_\gamma\bigr)n_\gamma.
 \label{eq:elastic-gradient-explicit}
\end{equation}
Because the normal bundle is one-dimensional,
$\nabla^\perp\widehat N_\gamma=0$.  Using
$\partial_s=2\partial_r$ in the normal Laplacian of the second identity in
\eqref{eq:Hopf-second-fundamental} gives
\[
 (\Delta^\perp H_{X_\gamma})\circ q_\gamma
 =8\partial_r^2k_\gamma\widehat N_\gamma.
\]
The first identity in \eqref{eq:Hopf-second-fundamental} also yields
$|(A_{X_\gamma})^\circ|^2\circ q_\gamma=2(k_\gamma^2+1)$.  Hence the
Willmore operator
$G=\frac12(\Delta^\perp H+|A^\circ|^2H)$ satisfies
\begin{equation}
 G_{X_\gamma}\circ q_\gamma
 =2\bigl(2\partial_r^2k_\gamma
                  +k_\gamma^3+k_\gamma\bigr)\widehat N_\gamma.
 \label{eq:surface-gradient-explicit}
\end{equation}
Combining \eqref{eq:normal-dilation},
\eqref{eq:elastic-gradient-explicit}, and
\eqref{eq:surface-gradient-explicit}, we obtain on $\R^2$
\begin{align*}
 \bigl(D\pi_{X_\gamma}(G_{X_\gamma})\bigr)\circ q_\gamma
 &=D\pi_{\widehat X_\gamma}
      (G_{X_\gamma}\circ q_\gamma)\\
 &=4\bigl(\nabla_{L^2}\El(\gamma)\bigr)
      \circ p_\gamma\circ q_\gamma.
\end{align*}
Surjectivity of $q_\gamma$ gives the pointwise identity
\begin{equation}
 \label{eq:intrinsic-projected-gradient}
 D\pi_{X_\gamma}(G_{X_\gamma})
 =4\bigl(\nabla_{L^2}\El(\gamma)\bigr)\circ p_\gamma
 \qquad\text{on }M_\gamma.
\end{equation}

The same covering calculation gives
$|G_{X_\gamma}|\circ q_\gamma
 =2|\nabla_{L^2}\El(\gamma)|\circ p_\gamma\circ q_\gamma$.  Integrating
over the same fundamental domain $\mathcal R$ as in 
\eqref{eq:energy-descent} therefore yields
\begin{align}
 \int_{M_\gamma}|G_{X_\gamma}|^2\dd\mu_{X_\gamma}
 &=2\pi\int_0^\ell
       4|\nabla_{L^2}\El(\gamma)|^2\dd s \notag\\
 &=4\pi\int_0^{L(\gamma)}
       |\nabla_{L^2}\El(\gamma)|^2\dd r.
 \label{eq:gradient-norm-descent}
\end{align}

The area measure, the Willmore energy density and the gradient pull back
by the local formulas above. Thus for the $d$-sheeted cover $Q$ one has
\[
 \Will(F)=d\Will(X_\gamma),
 \qquad
 G_F=G_{X_\gamma}\circ Q,
 \qquad
 \dd\mu_F=Q^*(\dd\mu_{X_\gamma}).
\]
Integration of a pulled-back density over a $d$-sheeted cover multiplies
its integral by $d$. Consequently, \eqref{eq:energy-descent},
\eqref{eq:intrinsic-projected-gradient}, and
\eqref{eq:gradient-norm-descent} give
\eqref{eq:W=piE}--\eqref{eq:gradient-norm-dictionary}. For $d=1$ these identities recover the projected gradient, energy and
dissipation formulas in \cite[arXiv version, (32)--(34)]{JakobHopf}.

For such a family set $u_t=p_{\gamma_t}\circ Q_t$.  Differentiating
$\pi\circ F_t=\gamma_t\circ u_t$ in time gives
\[
 D\pi_{F_t}(\partial_tF_t)
 =(\partial_t\gamma_t)\circ u_t+(D\gamma_t)_{u_t}(\partial_tu_t).
\]
The second term is tangent to the profile. The spatial derivative of the
same factorization gives $D\pi_{F_t}\circ DF_t=(D\gamma_t)_{u_t}\circ Du_t$.
Thus $D\pi_{F_t}$ also sends $(\partial_tF_t)^\top$ into the profile
tangent line. Projecting onto its orthogonal complement and using
$(\partial_tF_t)^\perp=-G_{F_t}$ and \eqref{eq:projected-gradient} gives
\[
 \bigl((\partial_t\gamma_t)^\perp\bigr)\circ u_t
 =-4\bigl(\nabla_{L^2}\El(\gamma_t)\bigr)\circ u_t.
\]
Surjectivity of $u_t$ gives \eqref{eq:profile-t}. The substitution
$\tau=4t$ gives \eqref{eq:standard-elastic-clock} in the same parametrization.
If $(\partial_t\gamma_t)^\top=D\gamma_t(X_t)$, composing $\gamma_t$ with
the flow of $-X_t$ on the parameter circle removes this tangential velocity.  Finally, \eqref{eq:speed-clock} is the
square root of \eqref{eq:gradient-norm-dictionary}, because
$\gamma_t=\bar\gamma_{4t}$.
\end{proof}

\begin{remark}[Clock normalization]
\label{rem:factor-four}
In Willmore time, \eqref{eq:W=piE} and
\eqref{eq:gradient-norm-dictionary} give
\[
 d\pi\frac{\dd}{\dd t}\El(\gamma_t)
 =\frac{\dd}{\dd t}\Will(F_t)
 =-4d\pi\|\nabla_{L^2}\El(\gamma_t)\|_{L^2(\dd s_{\gamma_t})}^2.
\]
This agrees with the pointwise equation \eqref{eq:profile-t} and fixes
the time conversion $\tau=4t$ independently of the covering degree.
In \cite{JakobHopf}, formula~(32) gives the factor $4$ in the projected
gradient. Accordingly, the right-hand side of formula~(51) in \cite{JakobHopf} must be multiplied by $4$ when the time variable 
is Willmore time.
\end{remark}

\section{Global and full convergence of the profile}
\label{sec:profile-convergence}

We first prove the parametrized reduction on the fixed initial cover.
The construction preserves its degree and does not require simplicity.

\begin{lemma}[Short-time uniqueness in the round sphere]
\label{lem:sphere-uniqueness}
For every smooth closed immersion $F_0\colon\Sigma\to\Sph^3$, the
Willmore equation \eqref{eq:WF} has at most one smooth solution with initial
datum $F_0$ on a common time interval.
\end{lemma}

\begin{proof}
Regard $F$ also as an immersion into $\R^4$.  In the present normalization
the Euclidean and spherical Willmore operators are, respectively,
\begin{align}
 G_F^{\R^4}
 &=\frac12\left(\Delta_{\R^4}^{\perp}H_F^{\R^4}
       +Q\bigl((A_F^{\R^4})^\circ\bigr)(H_F^{\R^4})\right),
 \label{eq:Euclidean-Willmore-operator}\\
 G_F^{\Sph^3}
 &=\frac12\left(\Delta_{\Sph^3}^{\perp}H_F^{\Sph^3}
       +Q\bigl((A_F^{\Sph^3})^\circ\bigr)(H_F^{\Sph^3})\right),
 \label{eq:spherical-Willmore-operator}
\end{align}
where
$Q(A^\circ)(\xi)=g^{ik}g^{jl}
 \langle A^\circ_{ij},\xi\rangle A^\circ_{kl}$.
With the trace convention used here,
$H_F^{\R^4}=H_F^{\Sph^3}-2F$, and hence
\[
 \frac14\int_\Sigma|H_F^{\R^4}|^2\dd\mu_F
 =\int_\Sigma\left(1+\frac14|H_F^{\Sph^3}|^2\right)\dd\mu_F.
\]
Thus the spherical functional is the restriction of the Euclidean Willmore
functional to sphere-valued immersions.  More precisely, the radial part of
$A_F^{\R^4}$ is pure trace, so
\begin{equation}
 (A_F^{\R^4})^\circ=(A_F^{\Sph^3})^\circ,
 \qquad
 \langle(A_F^{\Sph^3})^\circ,F\rangle=0.
 \label{eq:tracefree-ambient-identity}
\end{equation}
Moreover $\nabla_X^{\perp,\R^4}F=0$ for every $X\in T\Sigma$, and on normal
fields tangent to $\Sph^3$ the Euclidean and spherical normal connections
agree.  Consequently
\begin{align}
 \Delta_{\R^4}^{\perp}H_F^{\R^4}
 &=\Delta_{\Sph^3}^{\perp}H_F^{\Sph^3},
 \label{eq:normal-Laplacian-ambient-identity}\\
 Q\bigl((A_F^{\R^4})^\circ\bigr)(H_F^{\R^4})
 &=Q\bigl((A_F^{\Sph^3})^\circ\bigr)(H_F^{\Sph^3}).
 \label{eq:Q-ambient-identity}
\end{align}
Substitution in \eqref{eq:Euclidean-Willmore-operator}--
\eqref{eq:spherical-Willmore-operator} gives the pointwise identity
\begin{equation}
 G_F^{\R^4}=G_F^{\Sph^3},
 \qquad \langle G_F^{\R^4},F\rangle=0.
 \label{eq:ambient-Willmore-identity}
\end{equation}
In particular, every sphere-valued solution of \eqref{eq:WF} is also a
solution of the Euclidean Willmore equation in exactly the same time
normalization.

Proposition~1.1 of \cite{KuwertSchaetzle2002} gives uniqueness for the
Euclidean evolution equation
\[
 \partial_sf_s=-\bigl(\Delta^\perp H_{f_s}
                         +Q(A_{f_s}^\circ)(H_{f_s})\bigr).
\]
Its velocity is twice that in \eqref{eq:WF}: if $F_t$ solves
\eqref{eq:WF}, then $f_s:=F_{2s}$ solves the displayed equation.
Thus two solutions of \eqref{eq:WF} with the same parametrized initial
immersion agree on their common time interval.
\end{proof}

\begin{proposition}[Parametrized Hopf reduction]
\label{prop:parametrized-reduction}
Let $F_0$ satisfy the hypotheses of Theorem~\ref{thm:main}, and let
$\gamma_0$ and $q_0\colon\Sigma\to M_{\gamma_0}$ be supplied by
Proposition~\ref{prop:normalization}.  Let $\bar\gamma_\tau$ solve the
elastic initial value problem
\begin{equation}
 \label{eq:standard-profile-Cauchy}
 \partial_\tau\bar\gamma_\tau
 =-\nabla_{L^2}\El(\bar\gamma_\tau),
 \qquad \bar\gamma_0=\gamma_0,
\end{equation}
and put $\gamma_t:=\bar\gamma_{4t}$.  Then the Willmore flow with initial
datum $F_0$ is global, and for every $t\geq0$ there is a $d$-sheeted covering
\begin{equation}
 \label{eq:time-slice-normalization}
 Q_t\colon\Sigma\longrightarrow M_{\gamma_t}
 \quad\text{such that}\quad
 F_t=X_{\gamma_t}\circ Q_t.
\end{equation}
Consequently \eqref{eq:finite-time-multiplicity-main} holds.  The factorization applies to immersed and multiply covered profiles.
\end{proposition}

\begin{proof}
In the notation of \cite{DallAcquaEtAl2018},
\begin{equation}
 E_\lambda=\frac12\int_{\Sph^1}|\kappa|^2\dd s+\lambda L,
 \qquad \El=2E_{1/2},
 \label{eq:curve-energy-conventions}
\end{equation}
and therefore
$\nabla_{L^2}\El=2\nabla_{L^2}E_{1/2}$.  If
$\gamma^{\mathrm{DA}}_\rho$ denotes the flow in the time variable of that
paper, then $\bar\gamma_\tau=\gamma^{\mathrm{DA}}_{2\tau}$.  This factor-$2$
change comes solely from \eqref{eq:curve-energy-conventions}; it is distinct
from the factor-$4$ Willmore-to-elastic change $\tau=4t$.
The hypotheses of \cite[Theorem~1.1(i)]{DallAcquaEtAl2018} hold here: the
initial profile is a smooth regular closed curve in $\Sph^2$ and the length
coefficient is $\lambda=1/2>0$.  That theorem therefore shows that
\eqref{eq:standard-profile-Cauchy} exists smoothly for every $\tau\geq0$ and
remains regular.

Fix $T<\infty$ and identify the common parameter circle $B$ 
of the elastic flow with $\R/(2\pi\mathbb Z)$.  
Lift $t\mapsto\gamma_t(0)$ horizontally from a fixed point in its initial
Hopf fibre. For each $t$, take this lift as the initial value of the
horizontal lift $\eta_t$ of $s\mapsto\gamma_t([s])$. Both lifts depend
smoothly on their parameters by the horizontal-lift ODE. The holonomies
form a smooth circle-valued path on $[0,\infty)$, so fixing $h_0$ gives
a unique real lift $h_t$ on that half-line with
\begin{equation}
 \eta_t(s+2\pi)=e^{-ih_t}\eta_t(s) \quad \textnormal{for every} 
 \,\,t\in [0,T],
 \label{eq:finite-time-holonomy}
\end{equation}
as in \eqref{eq:horizontal-holonomy}. Set
\begin{equation}
 \Lambda_t:=\Span_{\mathbb Z}\{(2\pi,0),(h_t,2\pi)\},
 \qquad
 C_t:=\begin{pmatrix}2\pi&h_t\\0&2\pi\end{pmatrix},
 \qquad \Lambda_t=C_t\mathbb Z^2.
 \label{eq:finite-time-period-lattices}
\end{equation}
The period $2\pi$ here is that of the fixed profile parameter. The
Pinkall lattice in \eqref{eq:parametrized-lattice} uses the speed-$2$
parameter of period $L(\gamma_t)/2$.
The maps
\begin{equation}
 \mathcal I_t\colon\R^2/\Lambda_t\longrightarrow M_{\gamma_t},
 \qquad
 \mathcal I_t[\varphi,s]
 :=\bigl([s],e^{i\varphi}\eta_t(s)\bigr)
 \label{eq:finite-time-pullback-charts}
\end{equation}
are diffeomorphisms by the same deck-transformation argument as in
\eqref{eq:lattice-pullback-identification}.  Therefore
\begin{equation}
 \begin{alignedat}{2}
 \mathcal J_t&\colon\R^2/\Lambda_0\longrightarrow\R^2/\Lambda_t,
 &\qquad \mathcal J_t[x]&:=[C_tC_0^{-1}x],\\
 J_t&:=\mathcal I_t\circ\mathcal J_t\circ\mathcal I_0^{-1}
 \colon M_{\gamma_0}\longrightarrow M_{\gamma_t}
 \end{alignedat}
 \label{eq:finite-time-lattice-identification}
\end{equation}
are smooth diffeomorphisms.  Since
$C_tC_0^{-1}(\varphi,s)
=(\varphi+(h_t-h_0)s/(2\pi),s)$, the map $J_t$ covers the identity of the
profile circle:
\begin{equation}
 p_{\gamma_t}\circ J_t=p_{\gamma_0},
 \qquad J_0=\id.
 \label{eq:Jt-covers-identity}
\end{equation}
Define on the fixed torus $\Sigma$
\begin{equation}
 Y_t:=X_{\gamma_t}\circ J_t\circ q_0.
 \label{eq:constructed-Hopf-family}
\end{equation}
Then $Y_0=X_{\gamma_0}\circ q_0=F_0$.  
Moreover, on account of \eqref{eq:pullback-model} and 
\eqref{eq:Jt-covers-identity} we have
\begin{equation}
 \pi\circ Y_t
 =\gamma_t\circ p_{\gamma_0}\circ q_0.
 \label{eq:constructed-family-projection}
\end{equation}
Set $u_0:=p_{\gamma_0}\circ q_0$.  Differentiating
\eqref{eq:constructed-family-projection} with respect to $t$, and using that
$u_0$ is independent of $t$, gives
\[
 D\pi_{Y_t}(\partial_tY_t)=(\partial_t\gamma_t)\circ u_0.
\]
The normal component along the profile is therefore
\[
 \bigl(D\pi_{Y_t}(\partial_tY_t)\bigr)^{\perp_{\gamma_t}}
 =-4\bigl(\nabla_{L^2}\El(\gamma_t)\bigr)\circ u_0,
\]
where $\perp_{\gamma_t}$ denotes orthogonal projection in
$T_{\gamma_t(u_0(x))}\Sph^2$ onto the line perpendicular to
$D\gamma_t(T_{u_0(x)}\Sph^1)$.
The differential $D\pi$ annihilates the fibre direction, maps the horizontal
surface-tangent line onto the tangent line of $\gamma_t$, and maps the
surface-normal line in $T\Sph^3$ isomorphically onto the profile-normal line.
Consequently
\[
 D\pi_{Y_t}\bigl((\partial_tY_t)^\perp\bigr)
 =\bigl(D\pi_{Y_t}(\partial_tY_t)\bigr)^{\perp_{\gamma_t}}
 =D\pi_{Y_t}(-G_{Y_t}),
\]
where the last equality is \eqref{eq:projected-gradient}, applied with
$Q=J_t\circ q_0$ and using \eqref{eq:Jt-covers-identity}.
Injectivity on the surface-normal line gives
\begin{equation}
 (\partial_tY_t)^\perp=-G_{Y_t}.
 \label{eq:geometric-Willmore-family}
\end{equation}

Let $Z_t$ be the unique tangent vector field on $\Sigma$ determined by
\begin{equation}
 DY_t(Z_t)=-(\partial_tY_t)^\top,
\end{equation}
and let $\chi_t\in\Diff(\Sigma)$ solve
$\partial_t\chi_t=Z_t\circ\chi_t$, $\chi_0=\id$.  Compactness of $\Sigma$
gives this flow throughout $[0,T]$.  Set
\begin{equation}
 \widetilde F_t:=Y_t\circ\chi_t
 \quad\text{satisfies}\quad
 \partial_t\widetilde F_t=-G_{\widetilde F_t},
 \qquad \widetilde F_0=F_0.
\end{equation}
Indeed, the chain rule, the defining equation for $Z_t$, and naturality of
the Willmore gradient give the complete calculation
\[
 \partial_t(Y_t\circ\chi_t)
 =\bigl(\partial_tY_t+DY_t(Z_t)\bigr)\circ\chi_t
 =-G_{Y_t}\circ\chi_t
 =-G_{Y_t\circ\chi_t}.
\]

Let $T_{\max}$ be the maximal existence time of the solution $F_t$
of \eqref{eq:WF} from the theorem, and retain the arbitrary $T<\infty$ above.  By
Lemma~\ref{lem:sphere-uniqueness},
\begin{equation}
 F_t=\widetilde F_t
 \quad\text{for }0\leq t<\min\{T,T_{\max}\}.
 \label{eq:constructed-equals-maximal}
\end{equation}
If $T_{\max}<\infty$, choose $T>T_{\max}$.  The family $\widetilde F_t$ is a
smooth solution of \eqref{eq:WF} on all of $[0,T]$ and, by
\eqref{eq:constructed-equals-maximal}, agrees with $F_t$ before
$T_{\max}$.  It is therefore a smooth extension of the maximal solution
through $T_{\max}$, a contradiction.  Hence $T_{\max}=\infty$.  Since $T$
was arbitrary, \eqref{eq:constructed-equals-maximal} also proves
$F_t=\widetilde F_t$ for every $t\geq0$.  Formula
\eqref{eq:time-slice-normalization} follows with
\begin{equation}
 \label{eq:Qt-construction}
 Q_t:=J_t\circ q_0\circ\chi_t.
\end{equation}
Finally, give $F_t$ the parametrized Hopf data
$u_t:=p_{\gamma_t}\circ Q_t$ and $\gamma_t$.  Since a point of $M_{\gamma_t}$
is precisely a pair $(b,q)$, its canonical normalizing map from
\eqref{eq:normalizing-cover} is
\begin{equation}
 q_{F_t}(x)=\bigl(u_t(x),F_t(x)\bigr)=Q_t(x).
 \label{eq:time-slice-canonical-cover}
\end{equation}
It has the same degree $d$ as $q_0$, and the exact count
\eqref{eq:normalization-count} gives
\eqref{eq:finite-time-multiplicity-main}.
\end{proof}

\begin{proposition}[Full elastic-profile convergence]
\label{prop:profile-full}
Under the hypotheses of Theorem~\ref{thm:main}, the solution
$\bar\gamma_\tau$ of \eqref{eq:standard-profile-Cauchy} exists smoothly for
all $\tau\geq0$.  There are orientation-preserving diffeomorphisms
$\sigma_\tau\in\Diff(\Sph^1)$ and a smooth regular critical point
$\gamma_\infty$ of $\El$ such that
\begin{equation}
 \label{eq:profile-full}
 \bar\gamma_\tau\circ\sigma_\tau
 \longrightarrow\gamma_\infty
 \quad\text{in }C^k(\Sph^1,\R^3)
 \quad\text{for every }k\in\mathbb N_0.
\end{equation}
Here $\sigma_\tau$ is the unique orientation-preserving diffeomorphism
for which $\sigma_\tau(0)=0$ and $\bar\gamma_\tau\circ\sigma_\tau$
traverses its entire trace with constant speed. 
$\sigma_\tau$ depends smoothly on $\tau$.  With this choice
$\gamma_\infty$ is unique as a parametrized curve and may be multiply covered.
\end{proposition}

\begin{proof}
On $\Sph^1=\R/(2\pi\mathbb Z)$ define $\sigma_\tau$ as the inverse of
\[
 x\longmapsto \frac{2\pi}{L(\bar\gamma_\tau)}
          \int_0^x |\partial_y\bar\gamma_\tau(y)|\dd y.
\]
Regularity gives a smooth family of orientation-preserving diffeomorphisms.
Write $h_\tau=\bar\gamma_\tau\circ\sigma_\tau$, so that
$|\partial_x h_\tau|=L(\bar\gamma_\tau)/(2\pi)$.
The anchor gives $h_\tau(0)=\bar\gamma_\tau(0)$, a point that may move
in $\Sph^2$. Since $\El=2E_{1/2}$ and $\rho=2\tau$,
\cite[Theorem~1.1(ii) and its proof in \S4.2]{DallAcquaEtAl2018}
gives smooth precompactness of these constant-speed profiles along every
sequence of times tending to infinity. Choose $\tau_j\to\infty$ with
$h_{\tau_j}\to\gamma_\infty$ smoothly, where $\gamma_\infty$ is regular
and critical. By monotonicity,
$e(\tau):=\El(h_\tau)-\El(\gamma_\infty)\downarrow0$.

Equation~\eqref{eq:standard-profile-Cauchy} prescribes the full velocity
of $\bar\gamma_\tau$. Reparametrizing to $h_\tau$ adds a tangential
velocity. Viewed in $\R^3$, the curve has bending energy $\El/2$, since
$|\partial_s^2\bar\gamma_\tau|^2=1+|\kappa|^2$ in spherical arclength $s$.
The periodic Wirtinger inequality for the unit tangent gives
$L\El\ge4\pi^2$ as in \eqref{eq:profile-length-window}. Hence the
closed-curve comparison of \cite[Lemma~4.10 and Remark~4.11]{RuppSpener2024}
gives, with $C$ depending only on $\El(\gamma_0)$,
\begin{equation}
 \label{eq:anchored-profile-speed}
 \|\partial_\tau h_\tau\|_{L^2(\dd x)}
 \le C\|\partial_\tau\bar\gamma_\tau\|_{L^2(\dd s)}
 = C\|\nabla_{L^2}\El(\bar\gamma_\tau)\|_{L^2(\dd s)}.
\end{equation}

Choose $\delta>0$ so that the local gradient inequality of
Lemma~\ref{lem:LS-local} holds for every profile in the closed $C^4$-ball
of radius $\delta$ about $\gamma_\infty$. As long as $h_\tau$ stays in this
ball on $[a,b]$, that inequality and the energy identity
$-e'=\|\partial_\tau\bar\gamma_\tau\|_{L^2(\dd s)}^2$ give
\[
 \int_a^b\|\partial_\tau h_\tau\|_{L^2(\dd x)}\dd\tau
 \le C\bigl(e(a)^\theta-e(b)^\theta\bigr).
\]
Indeed, differentiate $e^\theta$ where $e>0$. If $e$ reaches zero, the energy identity gives
$\partial_\tau\bar\gamma_\tau=0$ thereafter.

We use the first-exit argument of \cite[\S4.4, pp.~22--23]{RuppSpener2024}
and \cite[pp.~25--27, arXiv formulas (76)--(80)]{JakobHopf}.
For large $j$, $h_{\tau_j}$ lies in the interior of the ball. Suppose that
for infinitely many such $j$ the profile reaches the boundary of the
closed $C^4$-ball about $\gamma_\infty$, and let $b_j>\tau_j$ be the
first such time. By smooth precompactness a subsequence of $h_{b_j}$
converges smoothly to a profile $h_*$ with
$\|h_*-\gamma_\infty\|_{C^4}=\delta$. The estimate on $[\tau_j,b_j]$ gives
\[
 \|h_{b_j}-h_{\tau_j}\|_{L^2(\dd x)}
 \le C e(\tau_j)^\theta\longrightarrow0.
\]
Since $h_{\tau_j}\to\gamma_\infty$, this forces $h_*=\gamma_\infty$,
a contradiction. Thus $h_\tau$ stays in the ball for all sufficiently
large $\tau$. The same estimate on $[a,b]$ now makes $h_\tau$ Cauchy in
$L^2(\dd x)$ as $a,b\to\infty$. Its $L^2$ limit is $\gamma_\infty$ by the
subsequence $h_{\tau_j}$. Every smooth subsequential limit therefore equals
$\gamma_\infty$, and smooth precompactness gives \eqref{eq:profile-full}.
\end{proof}

\begin{remark}[Scope]
The all-energy conclusion uses the preservation of the parametrized Hopf
factorization. It permits the reduction to the one-dimensional subcritical
elastic flow for every initial covering degree $d$.
\end{remark}

\section{\L ojasiewicz--Simon length and decay}
\label{sec:LS}

\begin{lemma}[Local gradient inequality]
\label{lem:LS-local}
For each smooth regular closed critical profile $\eta$ there are a
$C^4(\Sph^1,\Sph^2)$-neighborhood $U$ consisting of immersions and
constants $C>0$, $0<\theta\le1/2$, such that
\[
 |\El(\gamma)-\El(\eta)|^{1-\theta}
 \le C\|\nabla_{L^2}\El(\gamma)\|_{L^2(\dd s_\gamma)}
 \qquad (\gamma\in U).
\]
\end{lemma}

\begin{proof}
We apply the abstract inequality in \cite[Corollary~2.6]{PozzettaElastic}
directly to the present energy, including its length coefficient. Parametrize
$\eta$ by arclength on the circle of length $L(\eta)$, let $\nu$
be its oriented unit normal in $\Sph^2$, and define the graph energy
\[
 \Gamma_u:=\exp_{\eta}(u\nu),\qquad
 \mathcal F(u):=\El(\Gamma_u),
 \qquad u\in H^2_{\mathrm{per}},\quad \|u\|_{H^2}<r.
\]
On the unit sphere $\Gamma_u=\cos(u)\eta+\sin(u)\nu$. Its speed
$v=|\partial_s\Gamma_u|$ is bounded away from zero on a small $H^2$ ball
by $H^2\hookrightarrow C^1$. The signed curvature is
$\det(\Gamma_u,\partial_s\Gamma_u,\partial_s^2\Gamma_u)/v^3$.
Its numerator is affine in $u_{ss}$, with coefficients analytic in
$(u,u_s)$ and smooth in $s$. Multiplication $L^\infty\cdot L^2\to L^2$
and $L^2\cdot L^2\to L^1$ then give analyticity of the energy integral. Thus $D\mathcal F:H^2\to H^{-2}$ is analytic.

At $u=0$, $D\mathcal F=0$. Its Hessian is a symmetric fourth-order elliptic
operator with smooth periodic coefficients and leading part $2\partial_s^4$.
To see the leading term, the first variation of geodesic curvature has
leading term $u_{ss}$ and the quadratic curvature energy contributes
$2\int u_{ss}v_{ss}\,\dd s$ to the Hessian. The remaining terms have order
at most three. The operator $2\partial_s^4+1:H^2\to H^{-2}$ is an
isomorphism by Fourier series, while the lower-order terms are compact as
maps $H^2\to H^{-2}$. The Hessian is consequently Fredholm of index zero
and is symmetric under the $H^{-2},H^2$ pairing. Corollary~2.6, with
the reflexive spaces $V=Z=H^2$ and $W=(H^2)^*=H^{-2}$, gives numbers $C,r>0$ and
$\theta\in(0,1/2]$ such that
\[
 |\mathcal F(u)-\mathcal F(0)|^{1-\theta}
 \le C\|D\mathcal F(u)\|_{H^{-2}}
 \quad\text{for }\|u\|_{H^2}<r.
\]

We spell out the graph choice for an immersed critical curve. The map
$(s,z)\mapsto\exp_{\eta(s)}(z\nu(s))$ is a local diffeomorphism
near each point of the zero section. Compactness and the implicit function
theorem, used with the parameter circle retained, give a unique nearby
orientation-preserving reparametrization and a normal graph for any
parametrized curve sufficiently $C^1$-close to $\eta$; the local
solutions agree on overlaps. This remains true if distinct parameter values
have the same image, because each inverse is chosen near its own parameter
value. For curves in a sufficiently small $C^4$-neighborhood of $\eta$,
these graphs have $H^2$ norm less than $r$, by smoothness of the local
inverses and the chain rule.

For a $C^4$ graph and a smooth graph variation $h$, the geometric first
variation gives
\[
 D\mathcal F(u)[h]
 =\int\left\langle\nabla_{L^2}\El(\Gamma_u),
           \partial_u\Gamma_u\right\rangle h\,\dd s_{\Gamma_u}.
\]
The graph speed and $|\partial_u\Gamma_u|$ are uniformly bounded near
$u=0$. Cauchy--Schwarz, $H^2\hookrightarrow L^2$, and density of smooth
variations in $H^2$ give
\[
 \|D\mathcal F(u)\|_{H^{-2}}\le C
 \|\nabla_{L^2}\El(\Gamma_u)\|_{L^2(\dd s_{\Gamma_u})}.
\]
Reparametrization invariance of the energy and its geometric gradient norm
gives the local inequality at $\eta$.
\end{proof}

Set
\begin{equation}
 \label{eq:elastic-gap}
 \varepsilon(\tau):=
 \El(\bar\gamma_\tau)-\El(\gamma_\infty),
 \qquad
 g(\tau):=\|\nabla_{L^2}\El(\bar\gamma_\tau)\|_{L^2(\dd s)}.
\end{equation}
The energy identity in standard elastic time is
\begin{equation}
 \label{eq:elastic-dissipation}
 \varepsilon'(\tau)=-g(\tau)^2.
\end{equation}

\begin{corollary}[Elastic length and energy decay]
\label{cor:LS-tail}
For the limiting profile $\gamma_\infty$ of Proposition~\ref{prop:profile-full},
the exponent in Lemma~\ref{lem:LS-local} gives constants
$\tau_0\ge0$, $c,C>0$ such that
\begin{equation}
 \label{eq:LS-inequality}
 \varepsilon(\tau)^{1-\theta}\leq Cg(\tau)
 \qquad(\tau\geq\tau_0)
\end{equation}
and
\begin{equation}
 \label{eq:elastic-tail}
 \int_\tau^\infty g(s)\dd s
 \leq C\varepsilon(\tau)^\theta
 \qquad(\tau\geq\tau_0).
\end{equation}
Furthermore,
\begin{equation}
 \label{eq:elastic-rate}
 \varepsilon(\tau)\leq
 \begin{cases}
  Ce^{-c\tau},&\theta=\frac12,\\[1mm]
  C(1+\tau)^{-1/(1-2\theta)},&0<\theta<\frac12.
 \end{cases}
\end{equation}
\end{corollary}

\begin{proof}
Apply Lemma~\ref{lem:LS-local} at $\eta=\gamma_\infty$.
Proposition~\ref{prop:profile-full} puts its constant-speed representative
in $U$ for all sufficiently large $\tau$, giving \eqref{eq:LS-inequality}.

Where $\varepsilon>0$, \eqref{eq:elastic-dissipation} and
\eqref{eq:LS-inequality} imply
\begin{equation}
 \label{eq:LS-length-computation}
 -\frac{\dd}{\dd\tau}\varepsilon(\tau)^\theta
 =\theta\varepsilon(\tau)^{\theta-1}g(\tau)^2
 \geq c g(\tau).
\end{equation}
Integration proves \eqref{eq:elastic-tail}: if $\varepsilon$ reaches zero
at a finite time, both $\varepsilon$ and $g$ vanish thereafter by the energy
identity. Likewise,
\[
 \varepsilon'(\tau)=-g(\tau)^2
 \leq-c\varepsilon(\tau)^{2(1-\theta)}.
\]
Solving this ordinary differential inequality gives exponential decay when
$\theta=\frac12$ and the stated polynomial decay otherwise.
\end{proof}

\section{Covered equators: spectrum and attraction}
\label{sec:equators}

All calculations in this section concern the unweighted elastic energy
$\El$ on the unit sphere. Fix $m\geq1$ and put $\ell=2\pi m$.
We identify normal graphs over $\eta_m$ with real $\ell$-periodic functions
$u$ by
\begin{equation}
 \label{eq:equator-graph}
 \Gamma_u(s)=\cos u(s)\,\eta_m(s)+\sin u(s)\,(0,0,1).
\end{equation}
The normal is chosen so that at $u=0$ it equals $e_3:=(0,0,1)$.
Let $H^2_{\mathrm{per}}$ denote the real periodic Sobolev space, and
let $H^{-2}_{\mathrm{per}}$ be its dual under the $L^2(\dd s)$ pairing.

\begin{lemma}[Graph energy and spectral splitting]
\label{lem:equator-Hessian}
The graph energy $\mathcal F_m(u):=\El(\Gamma_u)$ is real analytic on a small $H^2_{\mathrm{per}}$ ball about $u=0$. Its first derivative at $u=0$ vanishes, and the Hessian of $\mathcal F_m$ at $u=0$ together with its spectrum are given by
\begin{equation}
	\label{eq:equator-Hessian}
	H_m=2\partial_s^4+3\partial_s^2+1
	\quad \textnormal{and} \quad
	\lambda_n=\left(\frac{n^2}{m^2}-1\right)
	\left(\frac{2n^2}{m^2}-1\right),\quad n\in\mathbb Z.
\end{equation}
The kernel of $H_m$ is $K=\operatorname{span}\{\cos s,\sin s\}$. 
Restricted to the $L^2$-orthogonal complement $H^2\cap K^\perp$, $H_m$ is an isomorphism onto the elements of $H^{-2}$ that annihilate $K$ under the dual pairing.  For $m=1,2,3$ there is $c_m>0$ such that
$\langle H_m u,u\rangle\geq c_m\|u\|_{H^2}^2$ for 
$u \in H^2\cap K^\perp$.  The number of negative eigenvalues of $H_m$, counted with real multiplicity, is given by \eqref{eq:Morse-index-main}.
\end{lemma}

\begin{proof}
Write $a=\cos u$, $b=\sin u$, $p=u_s$, and $v=(a^2+p^2)^{1/2}$.
Differentiating \eqref{eq:equator-graph} gives $|\partial_s\Gamma_u|=v$.
The determinant formula for signed geodesic curvature gives
\begin{equation}
 \label{eq:graph-curvature}
 k_u=\frac{a u_{ss}+b(a^2+2p^2)}{v^3},\qquad
 \dd s_{\Gamma_u}=v\dd s,
\end{equation}
and hence
\begin{equation}
 \label{eq:graph-energy}
 \mathcal F_m(u)=\int_0^\ell
 \left[v+\frac{\{a u_{ss}+b(a^2+2p^2)\}^2}{v^5}\right]\dd s.
\end{equation}
For completeness, in the orthonormal frame
$(\eta_m,\eta_m',e_3)$, the position, first derivative, and second
derivative have coordinates
$(a,0,b)$, $(-bp,a,ap)$, and
$(-ap^2-bu_{ss}-a,-2bp,-bp^2+au_{ss})$.
Their determinant is $a u_{ss}+b(a^2+2p^2)$, proving
\eqref{eq:graph-curvature} with the stated sign.

The embedding $H^2\hookrightarrow C^1$ in one dimension and $v\geq1/2$
on a sufficiently small ball make the coefficients analytic functions of
$(u,u_s)$. The expression is quadratic in $u_{ss}\in L^2$.
The multiplication estimates $\|fg\|_2\leq\|f\|_\infty\|g\|_2$
and $\|fg\|_1\leq\|f\|_2\|g\|_2$ show that its power series and
all its derivatives converge as maps on that ball. This proves analyticity
and in particular analyticity of $D\mathcal F_m:H^2\to H^{-2}$.
Expansion to second order gives
\begin{equation}
 \label{eq:graph-energy-expansion}
 \mathcal F_m(u)=\ell+
 \int_0^\ell\left(u_{ss}^2-\frac32u_s^2+\frac12u^2\right)\dd s
 +O(\|u\|_{H^2}^3).
\end{equation}
Indeed, $v=1+(u_s^2-u^2)/2+O((|u|+|u_s|)^4)$, and the linear part of
the curvature numerator is $u_{ss}+u$. Periodic integration by parts gives
$\int_0^\ell 2u u_{ss}\dd s=-2\int_0^\ell u_s^2\dd s$.
The remainder and its first derivative are bounded by
$C\|u\|_{H^2}^3$ and $C\|u\|_{H^2}^2$ in $\R$ and $H^{-2}$,
respectively, by the analytic expansion above.

The Fourier modes $e^{ins/m}$ diagonalize the Hessian of
$\mathcal F_m$ at $u=0$ and give \eqref{eq:equator-Hessian}. 
A zero eigenvalue requires $|n|=m$ or
$|n|=m/\sqrt2$; the latter cannot be an integer. Away from $n=\pm m$,
$|\lambda_n|/(1+(n/m)^2)^2$ has a positive lower bound, because it tends
to $2$ as $|n|\to\infty$ and the remaining set is a finite subset of $\mathbb{R}_{+}$. This proves the stated isomorphism by Fourier series. Negative eigenvalues occur precisely when $m/\sqrt2<|n|<m$. There are
$m-1-\lfloor m/\sqrt2\rfloor$ positive integers in that interval, each
giving a sine and a cosine. This interval contains no integer at all 
for the cases $m=1,2,3$. For $m\geq4$ it contains the integer $m-1$, because $m-1>m/\sqrt2$. This proves the asserted index formula \eqref{eq:Morse-index-main} and, by the same Fourier 
lower bound, coercivity of $H_m$ on $H^2\cap K^\perp$ for $m=1,2,3$.
\end{proof}

\begin{lemma}[Rotation slice and gradient inequality]
\label{lem:rotation-slice}
After a rotation and an orientation-preserving reparametrization, every
sufficiently small graph over $\eta_m$ has a graph representative
$w\in H^2\cap K^\perp$ with small norm. The first bound below holds for
all such representatives, and both bounds hold for smooth profiles.
\begin{equation}
 \label{eq:slice-gradient-bound}
 \|w\|_{H^2}\leq C_1\|D\mathcal F_m(w)\|_{H^{-2}}
 \leq C_2\|\nabla_{L^2}\El(\Gamma_w)\|_{L^2(\dd s_{\Gamma_w})}.
\end{equation}
For $m=1,2,3$, after shrinking the neighborhood,
\begin{equation}
 \label{eq:slice-coercivity}
 c\|w\|_{H^2}^2\leq\mathcal F_m(w)-\ell
 \leq C\|w\|_{H^2}^2.
\end{equation}
\end{lemma}

\begin{proof}
Use two rotation parameters $\alpha\in\R^2$ about horizontal axes, and
write $R_\alpha$ for the corresponding rotation with $R_0=I$.
For $u$ and $\alpha$ small, the longitude of $R_\alpha^{-1}\Gamma_u(s)$
has a unique lift $\phi_{\alpha,u}(s)$ close to $s$, with
$\phi(s+\ell)=\phi(s)+\ell$ and $\phi_s>1/2$.
Let $h_{\alpha,u}(s)=\arcsin((R_\alpha^{-1}\Gamma_u(s))_3)$.
The graph of the rotated curve is $w=h_{\alpha,u}\circ\phi_{\alpha,u}^{-1}$.
Its two moments are
\[
 \int_0^\ell h_{\alpha,u}(s)
 \begin{pmatrix}\cos\phi_{\alpha,u}(s)\\\sin\phi_{\alpha,u}(s)\end{pmatrix}
 \partial_s\phi_{\alpha,u}(s)\dd s.
\]
This formula is continuously differentiable in $(\alpha,u)\in\R^2\times H^2$:
its integrand is a smooth function of $\alpha,u,u_s$, with denominators
bounded away from zero. At $(0,0)$ the derivative with respect to the
rotation parameters is, up to signs and interchange of coordinates,
$\pi m I$. The implicit function theorem therefore selects
$\alpha(u)$ so that both moments vanish, with $|\alpha(u)|\leq C\|u\|_{H^2}$.
The first and second chain rules for $h\circ\phi^{-1}$, using
$\phi_s>1/2$, bound $\|w\|_{H^2}$ by $C(\|u\|_{H^2}+|\alpha|)$.
This argument uses only finite-dimensional rotation parameters; no
differentiability of composition on the full group of $H^2$
reparametrizations is being assumed. Rotations about the vertical axis
are absorbed by translations of the circle parameter. Compactness of the
rotation orbit makes the neighborhood sizes uniform along that orbit.

Let $P$ be the Fourier projection onto $K^\perp$, also acting boundedly
on $H^{-2}$. Taylor expansion of the analytic derivative gives
$P D\mathcal F_m(w)=H_m w+O(\|w\|_{H^2}^2)$.
The inverse estimate from Lemma~\ref{lem:equator-Hessian} and a smaller
radius absorb the quadratic term and prove the first inequality in
\eqref{eq:slice-gradient-bound}. For a variation $h$ of the graph,
the variational formula is
\[
 D\mathcal F_m(w)[h]=\int
 \langle\nabla_{L^2}\El(\Gamma_w),\partial_u\Gamma_w\rangle h
 \dd s_{\Gamma_w}.
\]
Since $|\partial_u\Gamma_w|=1$ and the arclength density is uniformly
bounded, Cauchy--Schwarz bounds this by
$C\|\nabla_{L^2}\El(\Gamma_w)\|_2\|h\|_{H^2}$, proving the second
inequality. Finally, coercivity of the Hessian on $K^\perp$ and the
cubic remainder in \eqref{eq:graph-energy-expansion} prove
\eqref{eq:slice-coercivity}.
\end{proof}

\begin{proof}[Proof of Theorem~\ref{thm:equator-rates}]
Rotation and reparametrization preserve the energy and the geometric
gradient norm. On the slice, Taylor's formula and
\eqref{eq:slice-gradient-bound} give
\[
 |\mathcal F_m(w)-\ell|^{1/2}\leq C\|w\|_{H^2}
 \leq C\|\nabla_{L^2}\El(\Gamma_w)\|_2.
\]
This proves \eqref{eq:optimal-LS-main}, even when the Hessian has negative
directions. If the gradient vanishes, \eqref{eq:slice-gradient-bound}
forces $w=0$, proving the description of the local critical set.
The exponent is optimal. For the constant graphs $u=a$ one has
$\mathcal F_m(a)=\ell/\cos a=\ell+\ell a^2/2+O(a^4)$ and
$\|\nabla_{L^2}\El(\Gamma_a)\|_2=\sqrt{\ell}\,|a|+O(|a|^3)$,
since $k=\tan a$ and $L=\ell\cos a$. An inequality with exponent $1-\theta<1/2$ on the energy
gap would therefore require $|a|^{2(1-\theta)}\leq C|a|$ as $a\to0$,
which is impossible.
For a trajectory converging modulo reparametrization to an $m$-covered
equator, set $\Delta(t)=\El(\gamma_t)-\ell\geq0$. For all sufficiently
large $t$ the profile lies in the neighborhood of \eqref{eq:optimal-LS-main}
modulo reparametrization. The energy identity then gives
\begin{equation}
 \label{eq:equator-gap-tail}
 \Delta'=-4\|\nabla_{L^2}\El(\gamma_t)\|_2^2\leq-c\Delta,
 \qquad
 \int_t^\infty\|\nabla_{L^2}\El(\gamma_s)\|_2\dd s
 \leq C\sqrt{\Delta(t)}.
\end{equation}
Integrate the differential inequality for $\Delta$ to obtain
$\Delta(t)\leq Ce^{-ct}$. The bound on the integrated gradient norm follows by differentiating
$\sqrt\Delta$ on $\{\Delta>0\}$ as in \eqref{eq:LS-length-computation}.
If $\Delta$ reaches zero, the gradient vanishes thereafter, and
\eqref{eq:standard-profile-Cauchy} gives a stationary profile.

Write $h_\tau=\bar\gamma_\tau\circ\sigma_\tau$ for the constant-speed
representative obtained from the reparametrization fixing $0$ in Proposition~\ref{prop:profile-full}.
Integrating \eqref{eq:anchored-profile-speed} from $4t$ to infinity and
using \eqref{eq:equator-gap-tail} gives
\[
 \|h_{4t}-\gamma_\infty\|_{L^2(\dd x)}
 \leq 4C\int_t^\infty\|\nabla_{L^2}\El(\gamma_s)\|_2\dd s
 \leq C'\sqrt{\Delta(t)}.
\]
Smooth convergence gives a uniform $H^N$ bound for each $N$. For
$N>k+1$, interpolation between $L^2$ and $H^N$, followed by
$H^{k+1}\hookrightarrow C^k$, yields
$\|h_{4t}-\gamma_\infty\|_{C^k}\leq C_k e^{-c_k t}$.

For the estimate of the profile length $L(\gamma_t)$, apply a fixed rotation and identify the
parameter circle with $\R/(2\pi m\mathbb Z)$ by a fixed linear change of
parameter and a shift so that $\gamma_\infty=\eta_m$. For large $t$ the
longitude of $h_{4t}$ has a lift $\phi_t$ close to the identity in $C^1$,
with $\phi_t(s+2\pi m)=\phi_t(s)+2\pi m$. Write
$h_{4t}(s)=\Gamma_{u(t)}(\phi_t(s))$. Since $\eta_m$ has zero height,
$|u(t,\phi_t(s))|=|\arcsin((h_{4t}(s))_3)|\leq C|h_{4t}(s)-\eta_m(s)|$.
Change variables $s\mapsto\phi_t(s)$ in this bound. Since $\phi_t'$ is uniformly bounded, it follows that
$\|u(t)\|_{L^2}\le C\sqrt{\Delta(t)}$. Smooth convergence gives
$u(t)\to0$ in $H^2$.

Fourier series in \eqref{eq:equator-Hessian} also give, without any
restriction to the slice,
\[
 \tfrac12\langle H_m u,u\rangle
 \geq c\|u\|_{H^2}^2-C\|u\|_2^2.
\]
Indeed, this follows for high frequencies from the positive leading
coefficient and for the finitely many remaining frequencies by increasing
$C$. Combining this estimate with the expansion
\eqref{eq:graph-energy-expansion} and absorbing its cubic remainder for
small $u$ proves
\begin{equation}
 \label{eq:fixed-graph-energy-control}
 \|u(t)\|_{H^2}^2\leq C\{\Delta(t)+\|u(t)\|_2^2\}
 \leq C\Delta(t).
\end{equation}
Since $L(\Gamma_u)=\int v\dd s$ and
$|v-1|\leq C(u^2+u_s^2)$, \eqref{eq:fixed-graph-energy-control}
proves \eqref{eq:quadratic-length-main}.

The exact transfer identity \eqref{eq:exact-length-transfer} and the
transport estimates \eqref{eq:TV-estimate}--\eqref{eq:BL-estimate}
give exponential decay of $L_F(t)$ and of both measure errors in
\eqref{eq:measure-rate-main}.
The marked-modulus tail estimate is proved in
Proposition~\ref{prop:classical-estimates}. For the surface rates in
the smooth reparametrized topology, use the construction in
Lemma~\ref{lem:anchored-lift} and Proposition~\ref{prop:smooth-lift}.
For the constant-speed profiles $h_{4t}$ the short-geodesic cylinder
in \eqref{eq:area-lift-local} has absolute area bounded by
$C\|h_{4t}-\gamma_\infty\|_{C^1}$. Transport along the short geodesic at
the base point gives an anchor difference bounded by
$C\|h_{4t}-\gamma_\infty\|_{C^0}$. In speed-$2$ parameters the change from
the anchored constant-speed parameter is linear, with slope determined
by the length. On a fixed compact interval the coefficients of the
horizontal-lift ODE and their spatial derivatives up to order $k$ therefore
differ from their limits by at most
$C_k\|h_{4t}-\gamma_\infty\|_{C^{k+1}}$. Subtracting the ODEs, applying Gronwall on
that fixed interval, and then differentiating them gives the same
exponential bounds for the horizontal lifts. The generator matrices and
their inverses have these bounds as well, since $L_\infty>0$.
Formula~\eqref{eq:explicit-anchored-convergence}, composed with the
fixed limiting cover, therefore gives exponential $C^k$ convergence of
the reparametrized surfaces.
\end{proof}

\begin{proof}[Proof of Theorem~\ref{thm:equator-stability}]
The index assertion is Lemma~\ref{lem:equator-Hessian}. For $m\geq4$,
take $u_\epsilon(s)=\epsilon\cos((m-1)s/m)$. Its negative Hessian
eigenvalue and \eqref{eq:graph-energy-expansion} imply
$\mathcal F_m(u_\epsilon)<\ell$ for every sufficiently small nonzero
$\epsilon$. Since the energy cannot increase, its global flow cannot
converge to any rotated $m$-covered equator, all of which have energy
$\ell$. This proves the claimed failure of an attracting neighborhood.
Now fix $m=1,2,3$.
Choose a radius $\rho>0$ so small that the rotation slice exists in a
larger neighborhood and \eqref{eq:slice-coercivity} holds uniformly for
all slice graphs of norm at most $2\rho$. Let $\mathcal U_r$ be the
set of curves admitting a rotation and a parametrization as such a slice
graph of norm less than $r$, for $0<r\leq2\rho$.
There are constants $c,C>0$ independent of the selected rotation such
that $c\|w\|_{H^2}^2\leq\El-\ell\leq C\|w\|_{H^2}^2$ in this
neighborhood. Choose $\delta<\rho/4$ with
$C\delta^2<c\rho^2/4$.

For a smooth initial profile in $\mathcal U_\delta$, global smooth
existence is already supplied by Proposition~\ref{prop:parametrized-reduction}
and its cited curve theorem. As long as the profile lies in
$\mathcal U_\rho$, monotonicity gives
\[
 c\|w(t)\|_{H^2}^2\leq\El(\gamma_t)-\ell
 \leq\El(\gamma_0)-\ell<C\delta^2<c\rho^2/4.
\]
It therefore admits a slice representative of norm less than $\rho/2$.
Suppose there were a finite first exit time $T$, and choose $t_j\uparrow T$
and rotations giving slice graphs $w_j$ of norm at most $\rho/2$.
Pass to a subsequence on which these rotations converge. The rotated
parametrized curves converge smoothly to a regular curve at $T$.
Write $\phi_j$ for lifts of their longitudes to the circle of length
$2\pi m$. In the original positively oriented parameter $x$,
\[
 \partial_x\phi_j=
 \frac{|\partial_x\gamma_{t_j}|}
 {\sqrt{\cos^2w_j+(\partial_s w_j)^2}\circ\phi_j}.
\]
Regularity at $T$ bounds the numerator away from zero, and the uniform
$C^1$ bound on $w_j$ bounds the denominator from above. The lifts can be
chosen with convergent initial values after passage to a subsequence.
Indeed, there are only $m$ choices modulo $2\pi m$, and all the rotated
curves lie in the same equatorial band. Smooth convergence in this band
gives smooth convergence of the lifted longitudes. The derivative bound
gives smooth convergence of their inverses on the lifted circle. Composing
the height functions with those inverses yields a slice graph at $T$ of
norm at most $\rho/2$. The two moments defining the slice pass to the
limit. By the rotation-slice construction the trajectory then continues
inside $\mathcal U_\rho$ past $T$, a contradiction.

Consequently the trajectory remains in this neighborhood for all time.
Proposition~\ref{prop:profile-full} gives a full smooth critical limit
modulo reparametrization. Applying the same compact-rotation and graph
argument along a sequence tending to infinity shows that this limit has
a slice representative of norm at most $\rho/2$. Lemma~\ref{lem:rotation-slice}
forces that representative to be zero. The limit is therefore a rotated
$m$-covered equator. Theorem~\ref{thm:equator-rates} proves the rates,
and the covering density formula proves total surface multiplicity $dm$.
\end{proof}

\section{Transfer of finite length to the surface}
\label{sec:surface-length}

\begin{proposition}[Finite Willmore metric length]
\label{prop:surface-length}
The solution $F_t$ of Theorem~\ref{thm:main} has finite $L^2$-length.
For $t\geq t_0:=\tau_0/4$, its length on $[t,\infty)$ satisfies
\begin{equation}
 \label{eq:exact-length-transfer}
 L_F(t)
 =\frac{\sqrt{d\pi}}{2}
 \int_{4t}^\infty g(\tau)\dd\tau
 \leq C\bigl(\Will(F_t)-\Will_\infty\bigr)^\theta.
\end{equation}
The energy estimates \eqref{eq:energy-rates-main} hold.
\end{proposition}

\begin{proof}
Since $F_t$ solves \eqref{eq:WF}, its velocity is perpendicular.
Formula \eqref{eq:speed-clock} and the substitution $\tau=4s$ give
\[
 \int_t^T\|\partial_sF_s\|_{L^2(\mu_s)}\dd s
 =\frac{\sqrt{d\pi}}{2}\int_{4t}^{4T}g(\tau)\dd\tau.
\]
Letting $T\to\infty$ and applying Corollary~\ref{cor:LS-tail} gives
\eqref{eq:exact-length-transfer}, since
\begin{equation}
 \label{eq:gap-transfer}
 \Will(F_t)-\Will_\infty=d\pi\varepsilon(4t).
\end{equation}
The length on $[0,t_0]$ is finite by smoothness.  Finally,
\eqref{eq:energy-rates-main} follows from
\eqref{eq:elastic-rate}, \eqref{eq:gap-transfer}, and the constant time
change.
\end{proof}

The finite-length conclusion uses the gradient inequality: the
square-integrability supplied by \eqref{eq:W-dissipation} alone would not
give an $L^1$ bound on an infinite time interval.

\section{Transport and the full varifold endpoint}
\label{sec:transport}

For a smooth family of immersions $F_t\colon\Sigma\to\Sph^3$, set
\begin{equation}
 \label{eq:transport-measures}
 \nu_t:=\mu_t\quad\text{on }\Sigma,
 \qquad
 \lambda_t:=(F_t)_\#\mu_t\quad\text{on }\Sph^3.
\end{equation}
For a finite signed measure $\rho$ on a compact space $K$, write
\begin{equation}
 \label{eq:TV-definition}
 \|\rho\|_{\TV}:=
 \sup_{\substack{\psi\in C(K)\\\|\psi\|_\infty\leq1}}
 \left|\int\psi\dd\rho\right|.
\end{equation}
For Radon measures on $\Sph^3$, set
\begin{equation}
 \label{eq:BL-definition}
 d_{\BL}(\alpha,\beta):=
 \sup_{\substack{\varphi\in C^1(\Sph^3)\\
 \|\varphi\|_\infty+\|\nabla\varphi\|_\infty\leq1}}
 \left|\int\varphi\dd\alpha-\int\varphi\dd\beta\right|.
\end{equation}

\begin{proposition}[Infinite-time transport criterion]
\label{prop:transport}
Suppose $F_t$ is defined for all $t\geq0$, $\sup_t\Will(F_t)\leq W_0$,
and $(\partial_tF_t)^\top=0$ at every point and time. Assume its
$L^2$-length is finite, so for each $t\geq0$,
\begin{equation}
 \label{eq:abstract-finite-length}
 \Lambda_F(t):=\int_t^\infty
 \|\partial_sF_s\|_{L^2(\mu_s)}\dd s<\infty.
\end{equation}
Then there are Radon measures $\nu_\infty$ and $\lambda_\infty$ such that
\begin{align}
 \|\nu_\infty-\nu_t\|_{\TV}
 &\leq2\sqrt{W_0}\,\Lambda_F(t),
 \label{eq:TV-estimate}\\
 d_{\BL}(\lambda_\infty,\lambda_t)
 &\leq2\sqrt{W_0}\,\Lambda_F(t).
 \label{eq:BL-estimate}
\end{align}
If $\Var_t$ denotes the integral two-varifold in $\R^4$ induced by $F_t$,
then $\Var_t$ converges along the full trajectory to a unique integral
rectifiable varifold $\Var_\infty$, and
$\|\Var_\infty\|=\lambda_\infty$.
\end{proposition}

\begin{proof}
For reference, the induced varifold is the Radon measure on
$\R^4\times\Gr(2,4)$, where $\Gr(2,4)$ is the Grassmannian of unoriented
two-planes in $\R^4$, defined by
\begin{equation}
 \label{eq:induced-varifold-definition}
 \Var_F(\Phi):=\int_\Sigma
 \Phi\bigl(F(x),DF_x(T_x\Sigma)\bigr)\dd\mu_F(x),
\qquad \Phi\in C_c(\R^4\times\Gr(2,4)).
\end{equation}
Its weight measure is $\|\Var_F\|=F_\#\mu_F$.
If $DF_t(X_t)=(\partial_tF_t)^\top$, the first variation of the area
density on the fixed domain has a tangential term
$(\operatorname{div}_{g_{F_t}}X_t)\dd\mu_t$. Here $X_t=0$, so
\begin{equation}
 \label{eq:area-evolution-transport}
 \partial_t\dd\mu_t
 =-\langle H_t,\partial_tF_t\rangle\dd\mu_t.
\end{equation}
Because
\begin{equation}
 \label{eq:spherical-controls}
 \Area(F_t)\leq W_0,
 \qquad
 \int_\Sigma|H_t|^2\dd\mu_t
 =4(\Will(F_t)-\Area(F_t))\leq4W_0,
\end{equation}
integration of \eqref{eq:area-evolution-transport} gives, for $u>t$,
\begin{equation}
 \label{eq:TV-Cauchy}
 \|\nu_u-\nu_t\|_{\TV}
 \leq2\sqrt{W_0}\int_t^u
 \|\partial_sF_s\|_{L^2(\mu_s)}\dd s.
\end{equation}
Thus $\nu_t$ is Cauchy in the Banach space of signed Radon measures with the
total-variation norm, and \eqref{eq:TV-estimate} follows.

For $\varphi\in C^1(\Sph^3)$,
\begin{align}
 \frac{\dd}{\dd t}\int_{\Sph^3}\varphi\dd\lambda_t
 =\int_\Sigma\bigl(&\langle\nabla^{\Sph^3}\varphi(F_t),
                              \partial_tF_t\rangle\notag\\
 &-\varphi(F_t)\langle H_t,\partial_tF_t\rangle\bigr)\dd\mu_t.
 \label{eq:image-evolution}
\end{align}
Writing $a=\|\varphi\|_\infty$ and
$b=\|\nabla\varphi\|_\infty$, Cauchy--Schwarz and
\eqref{eq:spherical-controls} bound the absolute value of
\eqref{eq:image-evolution} by
\[
 (b\sqrt{W_0}+2a\sqrt{W_0})
 \|\partial_tF_t\|_{L^2(\mu_t)}
 \leq2\sqrt{W_0}(a+b)
 \|\partial_tF_t\|_{L^2(\mu_t)}.
\]
This proves the bounded-Lipschitz Cauchy estimate.  The measures $\lambda_t$
are nonnegative, supported on the compact sphere, and have mass at most
$W_0$.  Weak compactness therefore supplies a Radon subsequential limit;
the Cauchy estimate makes this the full limit and gives
\eqref{eq:BL-estimate}.  Positivity is preserved for both $\nu_\infty$ and
$\lambda_\infty$.

It remains to control tangent planes.  View $\Var_t$ as a varifold in
$\R^4$.  If $H_t^{\R^4}$ is the Euclidean mean-curvature vector, then
\begin{equation}
 \label{eq:Euclidean-H}
 \int_\Sigma|H_t^{\R^4}|^2\dd\mu_t=4\Will(F_t)\leq4W_0.
\end{equation}
Consequently the masses and first variations of $\Var_t$ are uniformly
bounded.  Indeed, for every compactly supported $C^1$ vector field
$X$ on $\R^4$ with $\|X\|_\infty\leq1$,
\begin{equation}
 |\delta\Var_t(X)|
 \leq\left(\int_\Sigma|H_t^{\R^4}|^2\dd\mu_t\right)^{1/2}
      \Area(F_t)^{1/2}
 \leq2W_0.
 \label{eq:first-variation-uniform}
\end{equation}
The integral-varifold compactness theorem
\cite[Theorem~42.7 and Remark~42.8]{SimonGMT} gives an integral rectifiable subsequential limit.  Because all supports lie in the compact sphere, bounded-Lipschitz convergence in \eqref{eq:BL-estimate} is weak convergence of the weight measures.  Every subsequential varifold limit therefore has weight $\lambda_\infty$.

If an integral rectifiable
two-varifold has weight $\mu$, then
$\mu=\vartheta\mathcal H^2\llcorner M$ for a countably
two-rectifiable set $M$ and an integer-valued multiplicity $\vartheta$.
Both the approximate tangent plane $T_xM$ and the density $\vartheta(x)$ are
determined by $\mu$ for $\mu$-almost every $x$.  Thus the varifold itself,
which integrates at $(x,T_xM)$ with weight $\vartheta$, is determined by
$\mu$.  It follows that all subsequential limits above are the same
varifold $\Var_\infty$.

Every sequence of times tending to infinity thus has a subsequence
converging to $\Var_\infty$. Weak convergence on the compact space
$\Sph^3\times\Gr(2,4)$ is metrizable on measures of bounded mass,
so the entire trajectory converges.
\end{proof}

For the solution of \eqref{eq:WF}, $\Lambda_F=L_F$, which proves the
measure estimates in Theorem~\ref{thm:endpoint}\textnormal{(ii)}.
Under $F_t\mapsto F_t\circ\Psi_t$, the image measure and varifold are
unchanged, while the measure on $\Sigma$ becomes $\Psi_t^*\mu_t$.
The total-variation estimate on $\Sigma$ uses the parametrization of
\eqref{eq:WF}.

\section{Anchored lifting and Pinkall modulus}
\label{sec:lifting}

For a smooth regular closed curve $\gamma$ of length $L$, let $C$ be an integral two-current in $\Sph^2$ with $\partial C=\gamma_\#[\Sph^1]$.  
Its oriented area is well-defined modulo $4\pi$ by
\begin{equation}
 \label{eq:area-class}
 [\mathcal A(\gamma)]\in\R/(4\pi\mathbb Z),
 \qquad
 \mathcal A(\gamma)=C(\vol_{\Sph^2}).
\end{equation}
Such a filling exists because $H_1(\Sph^2;\mathbb Z)=0$.  Two fillings
differ by an integral two-cycle in $\Sph^2$, hence by an integer multiple of
the fundamental current.  This definition includes self-intersecting and
multiply traversed profiles.  Identify the Lie algebra of the fibre circle
with $\R$ through $a\mapsto ia$.  The connection one-form is
$\alpha_q(v):=\langle v,iq\rangle$, and its kernel is the horizontal
space $\mathcal H_q$.  Its curvature satisfies
$\mathrm d\alpha=\pi^*\Omega$, where
\begin{equation}
 \label{eq:Hopf-curvature}
\Omega=\frac12\vol_{\Sph^2}.
\end{equation}
For the sign, at $q=1$ take horizontal vectors $j,k$.  Directly from
$\alpha=\langle\mathrm dq,iq\rangle$ one obtains
$\mathrm d\alpha(j,k)=2$.  Also $D\pi_1(j)=2k$ and
$D\pi_1(k)=-2j$, and hence 
$(\pi^*\vol_{\Sph^2})_1(j,k)=4$ according to the orientation we have fixed in Section~\ref{sec:dictionary}.  Right multiplication on $\Sph^3$ preserves
$\alpha$ and induces the rotation $y\mapsto r^{-1}yr$ on $\Sph^2$;
this proves \eqref{eq:Hopf-curvature} everywhere.
To fix the holonomy sign and the use of Stokes' theorem, first let
$U\subset\Sph^2$ admit a smooth section $\sigma$, put
$b=\sigma^*\alpha$, and write
$\Phi(x,e^{i\theta})=e^{i\theta}\sigma(x)$. Fibre invariance gives
\[
 \Phi^*\alpha=\dd\theta+b,\qquad \dd b=\tfrac12\vol_{\Sph^2}|_U.
\]
In particular we infer that $b$ need not be closed and that a horizontal lift, written locally as $e^{i\theta(s)}\sigma(\gamma(s))$, satisfies
$\theta'=-b(\gamma')$.

Because $\Sph^2$ is simply connected, the smooth parametrized loop
$\gamma$ bounds a smooth map $u:D^2\to\Sph^2$. The pullback bundle over
the disc is trivial; choose a smooth lift $\widehat u:D^2\to\Sph^3$ and
put $\beta=\widehat u^*\alpha$. Then $\dd\beta=u^*\vol_{\Sph^2}/2$.
Writing the horizontal lift along the oriented boundary as
$\eta=e^{i\theta}\widehat u$ and integrating the horizontal equation gives
\[
 \theta(1)-\theta(0)
 =-\int_{\partial D^2}\beta
 =-\frac12\int_{D^2}u^*\vol_{\Sph^2}.
\]
The disc current
$u_\#[D^2]$ is one of the fillings in \eqref{eq:area-class}; the
corresponding area is therefore $\mathcal A(\gamma)$ modulo $4\pi$.
In the speed-$2$ parameter of period $L/2$, the horizontal lift therefore
satisfies
\begin{equation}
 \label{eq:holonomy-current}
 \eta(s+L/2)=e^{-i\mathcal A(\gamma)/2}\eta(s).
\end{equation}
Changing the filling by $m[\Sph^2]$, $m\in\mathbb Z$, changes the
half-area by $2\pi m$ and leaves the holonomy unchanged. Pinkall's
formula \cite[Proposition~1]{Pinkall1985}, with the deck convention
\eqref{eq:horizontal-holonomy}, gives for every parametrized closed immersion
\begin{equation}
 \label{eq:Pinkall-lattice}
 \Gamma_\gamma=
 \Span_{\mathbb Z}\left\{(2\pi,0),
                 \left(\frac{\mathcal A(\gamma)}2,\frac{L(\gamma)}2\right)\right\}.
\end{equation}
Adding $4\pi$ to the area representative adds the first generator to the
second. Reversing the profile orientation with the fibre orientation
fixed replaces $z=(\mathcal A+iL)/(4\pi)$ by $-\overline z$ for the
opposite area representative.

\begin{lemma}[Anchored Hopf-lift convergence]
\label{lem:anchored-lift}
Let $\gamma_j,\gamma_\infty\colon\Sph^1\to\Sph^2$ be smooth regular closed
curves such that $\gamma_j\to\gamma_\infty$ in $C^\infty$.  Let
$X_j\colon M_j:=M_{\gamma_j}\to\Sph^3$ be the standard parametrized models
\eqref{eq:standard-Hopf-model}, and write
$M_\infty:=M_{\gamma_\infty}$ and
$X_\infty:=X_{\gamma_\infty}$.  No simplicity is assumed.  Then:
\begin{enumerate}[label=\textnormal{(\alph*)},leftmargin=8mm]
 \item after choosing a continuous lift of the area class,
 $L_j\to L_\infty$ and $\mathcal A_j\to \mathcal A_\infty$;
 \item the marked lattices $\Gamma_j$ in \eqref{eq:Pinkall-lattice}
 converge to $\Gamma_\infty$ in the sense that their marked generator
 matrices converge in $M_2(\R)$;
 \item there are diffeomorphisms
 $K_j\colon M_\infty\to M_j$ such that
 \begin{equation}
  \label{eq:anchored-model-convergence}
  X_j\circ K_j\longrightarrow X_\infty
  \quad\text{in }C^k(M_\infty,\R^4)
  \quad\text{for every }k\in\mathbb N_0;
 \end{equation}
 \item if $\widehat\gamma=\gamma\circ\sigma$ for an orientation-preserving
 diffeomorphism $\sigma$ of $\Sph^1$, then there is a diffeomorphism
 $R_\sigma\colon M_\gamma\to M_{\widehat\gamma}$ such that
 \begin{equation}
  \label{eq:profile-reparametrization-lift}
  X_\gamma=X_{\widehat\gamma}\circ R_\sigma.
 \end{equation}
\end{enumerate}
\end{lemma}

\begin{proof}
Regularity of $\gamma_\infty$ and $C^1$ convergence give, in the fixed
parameter $x$ on $\Sph^1$,
\[
 c_0:=\min_{\Sph^1}|\partial_x\gamma_\infty|>0,
 \qquad
 \min_{\Sph^1}|\partial_x\gamma_j|\geq c_0/2
 \quad\text{for all sufficiently large }j.
\]
The anchored arclength reparametrizations depend smoothly on a regular
curve.  Their inverses, extended to the universal covers, therefore converge
smoothly on compact intervals, and $L_j\to L_\infty>0$.
We use the speed-$2$ parametrizations with periods $L_j/2$ in the horizontal
lift construction below.  For the area, take the curves in their original
common parameter sufficiently close
to $\gamma_\infty$ and join $\gamma_\infty(s)$ to $\gamma_j(s)$ by the
unique short geodesic.  Orient the resulting thin cylinder $C_j$ so that
$\partial C_j=(\gamma_j)_\#[\Sph^1]-(\gamma_\infty)_\#[\Sph^1]$.  
It gives the unambiguous local lift
\begin{equation}
 \label{eq:area-lift-local}
 \mathcal A_j-\mathcal A_\infty:=\int_{C_j}\vol_{\Sph^2},
\end{equation}
which tends to zero by smooth convergence.  Formula
\eqref{eq:Pinkall-lattice} then proves lattice convergence.

For the lifted statement, fix a point
$q_\infty\in\pi^{-1}(\gamma_\infty(0))$ and choose anchors
$q_j\in\pi^{-1}(\gamma_j(0))$ with $q_j\to q_\infty$.  The horizontal-lift
equation is a smooth first-order ODE.  Continuous dependence, differentiated
in the curve parameter, implies $C^\infty$ convergence on every compact
subset of the universal cover of the anchored horizontal lifts $\eta_j$.
With the convention \eqref{eq:horizontal-holonomy}, their endpoint
holonomies are $e^{-i\mathcal A_j/2}$. These lifts determine the universal covering maps
\[
 (\varphi,s)\longmapsto
 \bigl(a_{\gamma_j}([s]),e^{i\varphi}\eta_j(s)\bigr)\in M_j
\]
with deck groups $\Gamma_j$.
Let
$I_j\colon\R^2/\Gamma_j\to M_j$ and
$I_\infty\colon\R^2/\Gamma_\infty\to M_\infty$ be the diffeomorphisms
\eqref{eq:lattice-pullback-identification} determined by these choices.

More explicitly, let
\begin{equation}
 B_j:=\begin{pmatrix}2\pi&\mathcal A_j/2\\0&L_j/2\end{pmatrix},
 \qquad
 B_\infty:=\begin{pmatrix}2\pi&\mathcal A_\infty/2\\0&L_\infty/2\end{pmatrix},
 \qquad P_j:=B_jB_\infty^{-1}.
 \label{eq:anchored-generator-matrices}
\end{equation}
Here $\det B_\infty=\pi L_\infty>0$, so $B_\infty^{-1}$ is well defined.
Then $\Gamma_j=B_j\mathbb Z^2$, $P_j\Gamma_\infty=\Gamma_j$, and
$P_j\to I$ in $M_2(\R)$.  Therefore
\begin{equation}
 \widetilde K_j\colon\R^2/\Gamma_\infty
 \longrightarrow\R^2/\Gamma_j,
 \qquad \widetilde K_j[x]:=[P_jx],
 \label{eq:explicit-Kj}
\end{equation}
is a well-defined diffeomorphism, and
\begin{equation}
 K_j:=I_j\circ\widetilde K_j\circ I_\infty^{-1}
 \colon M_\infty\longrightarrow M_j
 \label{eq:intrinsic-Kj}
\end{equation}
is the required identification of the intrinsic pullback tori.  Write
$P_j(\varphi,s)=(\varphi_j(\varphi,s),s_j(\varphi,s))$ on the universal
cover.  The pulled-back immersion has the explicit formula
\begin{equation}
 (X_j\circ K_j\circ I_\infty)[\varphi,s]
 =e^{i\varphi_j(\varphi,s)}\eta_j\bigl(s_j(\varphi,s)\bigr).
 \label{eq:explicit-anchored-convergence}
\end{equation}
On a compact fundamental parallelogram for $\Gamma_\infty$, the arguments
$P_j(\varphi,s)$ remain in one fixed compact set.  The already established
$C^\infty_{\mathrm{loc}}$ convergence $\eta_j\to\eta_\infty$, together with
$P_j\to I$, therefore permits differentiation of
\eqref{eq:explicit-anchored-convergence} to every order and yields
\[
 e^{i\varphi_j}(\eta_j \circ s_j)
 \longrightarrow e^{i\varphi}\eta_\infty(s)
 \quad\text{in }C^\infty.
\]
The equivariance under $\Gamma_\infty$ makes this convergence well defined
on $\R^2/\Gamma_\infty$.  Since $I_\infty$ is a fixed diffeomorphism and
$(X_\infty\circ I_\infty)[\varphi,s]
=e^{i\varphi}\eta_\infty(s)$ by \eqref{eq:standard-Hopf-model}, 
it is exactly \eqref{eq:anchored-model-convergence} on $M_\infty$.

Finally, if $\widehat\gamma=\gamma\circ\sigma$, 
the intrinsic pullback model gives the explicit diffeomorphism
\begin{equation}
 R_\sigma\colon M_\gamma\longrightarrow M_{\widehat\gamma},
 \qquad R_\sigma(b,q):=(\sigma^{-1}(b),q).
 \label{eq:explicit-profile-reparametrization}
\end{equation}
It satisfies $X_\gamma=X_{\widehat\gamma}\circ R_\sigma$ by definition.
This reparametrization map $R_\sigma$ is defined without choosing
horizontal lifts and does not require either profile to be injective.
The convergence identifications $K_j$ constructed above, in contrast, use
the chosen anchored horizontal lifts.
\end{proof}

\begin{proposition}[Smooth surface and modulus convergence]
\label{prop:smooth-lift}
The profile convergence in Proposition~\ref{prop:profile-full} lifts to
\eqref{eq:smooth-main}.  The limiting immersion retains covering degree $d$, is Willmore, and
\eqref{eq:pinkall-main}--\eqref{eq:modulus-main} hold.  Moreover, the
varifold obtained in Proposition~\ref{prop:transport} is
$\Var_{F_\infty}$.
\end{proposition}

\begin{proof}
Let $\sigma_\tau$ be the orientation-preserving diffeomorphism of $\Sph^1$
that fixes $0$ and gives $\bar\gamma_\tau\circ\sigma_\tau$ constant speed.
By Proposition~\ref{prop:profile-full}, $\sigma_\tau$ depends smoothly on
$\tau$, and the reparametrized profiles
\begin{equation}
 \widehat\gamma_t:=\gamma_t\circ\sigma_{4t}
                  =\bar\gamma_{4t}\circ\sigma_{4t}
\end{equation}
converge smoothly to the regular profile $\gamma_\infty$. Write
$\widehat M_t:=M_{\widehat\gamma_t}$ and
$\widehat X_t:=X_{\widehat\gamma_t}$. We apply
Lemma~\ref{lem:anchored-lift}\textnormal{(c)} with anchors that vary
smoothly in $t$. Fix $\xi_\infty\in\pi^{-1}(\gamma_\infty(0))$.
For all sufficiently large $t$, horizontal transport along the unique short
geodesic from $\gamma_\infty(0)$ to $\widehat\gamma_t(0)$ gives an anchor
$\xi_t$ depending smoothly on $t$ and tending to $\xi_\infty$.
Choose the real area $\mathcal A(t)$ by the short-geodesic cylinder
in \eqref{eq:area-lift-local}. Its area integral and the profile length
depend smoothly on $t$. Their limits in Lemma~\ref{lem:anchored-lift}\textnormal{(a)}
give $B(t)\to B_\infty$ for the generator matrices in
\eqref{eq:anchored-generator-matrices}.

With these anchors, the horizontal-lift ODE gives spatial lifts depending
smoothly on $t$ and converging in $C^\infty$ on compact intervals of the
universal cover. In the resulting lattice charts,
\eqref{eq:explicit-Kj}--\eqref{eq:intrinsic-Kj} define the diffeomorphism
$K_t$ by the matrix $B(t)B_\infty^{-1}$. Hence $K_t$ depends smoothly on $t$. Since the matrix
tends to the identity, \eqref{eq:explicit-anchored-convergence} on a fixed
fundamental parallelogram gives
\begin{equation}
 K_t\colon M_{\gamma_\infty}\longrightarrow\widehat M_t,\qquad
 \widehat X_t\circ K_t\longrightarrow X_{\gamma_\infty}
 \quad\text{in }C^\infty(M_{\gamma_\infty},\R^4).
 \label{eq:Kt-lift}
\end{equation}

The reparametrization $\sigma_{4t}$ also induces a diffeomorphism of
pullback tori. By \eqref{eq:explicit-profile-reparametrization}, it is given by
\begin{equation}
 R_t\colon M_{\gamma_t}\longrightarrow\widehat M_t,\qquad
 R_t(b,q)=(\sigma_{4t}^{-1}(b),q),\qquad
 X_{\gamma_t}=\widehat X_t\circ R_t.
 \label{eq:Rt-lift}
\end{equation}
This formula gives smooth dependence on $t$ directly.
Recall from Proposition~\ref{prop:parametrized-reduction} that
$F_t=X_{\gamma_t}\circ Q_t$ by \eqref{eq:time-slice-normalization}, where
$Q_t=J_t\circ q_0\circ\chi_t$ in \eqref{eq:Qt-construction} is a smooth
family of $d$-sheeted coverings. Since $R_t$ and $K_t$ are diffeomorphisms,
\[
 p_t:=K_t^{-1}\circ R_t\circ Q_t\colon\Sigma\longrightarrow M_{\gamma_\infty}
\]
is a smooth family of $d$-sheeted coverings of the fixed limit torus.
Fix a sufficiently large $t_*$ and put $q_\infty:=p_{t_*}$.
The time-dependent vector field $Z_t:=-(Dp_t)^{-1}(\partial_t p_t)$
on $\Sigma$ is smooth since each $p_t$ is a local diffeomorphism.
Its resulting flow $\Psi_t$ starting at $t_*$ exists for all 
$t \geq t_*$ by compactness of $\Sigma$. Now the chain rule gives
$\partial_t(p_t\circ\Psi_t)=(\partial_t p_t+Dp_t Z_t)\circ\Psi_t=0$ 
for all $t \geq t_*$, and hence we obtain: 
\begin{equation}
 \label{eq:explicit-Psi}
 p_t\circ\Psi_t=q_\infty,\qquad \Psi_{t_*}=\id.
\end{equation}
Thus, under the identifications $K_t^{-1}\circ R_t$, the covers
\eqref{eq:time-slice-normalization} determine the same index-$d$ subgroup
$q_{\infty*}\pi_1(\Sigma)\subset\pi_1(M_{\gamma_\infty})\cong\mathbb Z^2$.
Changing basepoints acts by conjugation, which is trivial in this abelian group.

Now define $F_\infty:=X_{\gamma_\infty}\circ q_\infty$. Combining
\eqref{eq:time-slice-normalization} and \eqref{eq:Rt-lift} with the
definition of $p_t$ and \eqref{eq:explicit-Psi} gives
\[
 F_t\circ\Psi_t
 =\widehat X_t\circ R_t\circ Q_t\circ\Psi_t
 =\widehat X_t\circ K_t\circ q_\infty .
\]
Composing \eqref{eq:Kt-lift} with the fixed smooth cover $q_\infty$ now yields
\[
 F_t\circ\Psi_t\longrightarrow F_\infty
 \quad\text{in }C^k(\Sigma,\R^4)\quad\text{for every }k\in\mathbb N_0 .
\]
The diffeomorphisms can be chosen smoothly for all $t\ge0$. Indeed, take a
smooth function $\beta\colon[0,\infty)\to[t_*,\infty)$ equal to $t_*$ on
$[0,t_*]$ and to $t$ for $t\ge t_*+1$, and use $\Psi_{\beta(t)}$.
This gives \eqref{eq:smooth-main}.

In the lattice coordinates of Lemma~\ref{lem:anchored-lift}, the subgroup
of index $d$ is $M\mathbb Z^2\subset\mathbb Z^2$ for a fixed integral
matrix $M$ with positive determinant $d$. The covered torus therefore has
period lattice $B(t)M\mathbb Z^2$. The same linear identifications
$B(t)B_\infty^{-1}$ used in that lemma map the covered limit lattice to
the covered time-$t$ lattice. If $M=\left(\begin{smallmatrix}a&c\\b&e\end{smallmatrix}\right)$,
the complex generators are $2\pi(a+bz)$ and $2\pi(c+ez)$.
Their quotient is $z_M$ in \eqref{eq:modulus-main}, and
\[
 \operatorname{Im}z_M=\frac{d\operatorname{Im}z}{|a+bz|^2}>0.
\]
This proves the covered marking formula, including its orientation and
degree factor.
If the area representative is changed to $\mathcal A+4\pi n$, then
$B$ is replaced by $BT_n$, where
$T_n=\left(\begin{smallmatrix}1&n\\0&1\end{smallmatrix}\right)
\in\mathrm{SL}_2(\mathbb Z)$. Hence $M_n=T_n^{-1}M$ is integral with
$\det M_n=\det M=d>0$. The identity $(BT_n)M_n=BM$ shows that the
covered marked modulus is unchanged when the
area representative and the subgroup coordinates are transformed together.

Since
$\nabla_{L^2}\El(\gamma_\infty)=0$, the projected-gradient identity
\eqref{eq:projected-gradient}, together with the fact that $D\pi$ is
injective when restricted to the normal line by \eqref{eq:normal-dilation}, yields $G_{F_\infty}=0$.  Hence $F_\infty$ is
Willmore.  Smooth convergence gives \eqref{eq:limit-energy-main} and, because varifolds are invariant under domain diffeomorphisms,
\begin{equation}
 \label{eq:smooth-implies-varifold}
 \Var_t=\Var_{F_t\circ\Psi_t}
 \rightharpoonup\Var_{F_\infty}.
\end{equation}
Uniqueness in Proposition~\ref{prop:transport} identifies this limit with
$\Var_\infty$.  Finally, the first two parts of
Lemma~\ref{lem:anchored-lift} give convergence of the lattices and their
modular classes.  Since $\sigma_{4t}$ is orientation preserving, it changes
neither profile length nor oriented area class; hence these are precisely the
quantities attached to $\gamma_t$ in Theorem~\ref{thm:endpoint}\textnormal{(iv)}.
\end{proof}

\section{Sharp length and quantitative conformal estimates}
\label{sec:quantitative-modulus}

Fix a solution $F_t$ of \eqref{eq:WF} as in Theorem~\ref{thm:main}
and its reduced profile $\gamma_t$, $t\geq0$. All time derivatives below
are taken in Willmore time, so $(\partial_t\gamma_t)^\perp
=-4\nabla_{L^2}\El(\gamma_t)$ by \eqref{eq:profile-t}. Set
\begin{equation}
 \label{eq:classical-dissipation-notation}
 E(t)=\El(\gamma_t),\quad L(t)=L(\gamma_t),\quad
 \Gcal_t=2\partial_r^2k_t+k_t^3+k_t,\quad
 \Dcal(t)=4\int\Gcal_t^2\dd s_{\gamma_t}=-E'(t).
\end{equation}
Here $r$ denotes spherical arclength, $k_t$ denotes signed geodesic curvature,
and $\nabla_{L^2}\El=\Gcal_t n_t$. Thus the surface normalization is
\begin{equation}
 \label{eq:covered-dissipation}
 \Will(F_t)=d\pi E(t),\qquad
 -\frac{\dd}{\dd t}\Will(F_t)=d\pi\Dcal(t),\qquad
 \|\partial_tF_t\|_2=\sqrt{d\pi\Dcal(t)}.
\end{equation}

\begin{proposition}[Classical length and marked-modulus control]
\label{prop:classical-estimates}
Let $E_0=E(0)$ and let $\mathcal A(t)$ be a continuous real lift of the
oriented area class of $\gamma_t$. Along this trajectory,
\begin{align}
 (L')^2&\leq4(E-L)\Dcal,
 \label{eq:sharp-classical-length}\\
 (\mathcal A')^2&\leq4L\Dcal,
 \label{eq:classical-area-bound}\\
 (\mathcal A')^2+(L')^2&\leq4\lambda_+(t)\Dcal
 \leq4E\Dcal,
 \label{eq:joint-classical-bound}
\end{align}
where, with $C_1(t)=\int k_t\dd s_{\gamma_t}$,
\begin{equation}
 \label{eq:moment-eigenvalue}
 \lambda_+(t)=\frac{E+\sqrt{(2L-E)^2+4C_1^2}}2.
\end{equation}
The coefficient $4$ in \eqref{eq:sharp-classical-length} is sharp.
Furthermore,
\begin{equation}
 \label{eq:profile-length-window}
 \frac{4\pi^2}{E_0}\leq\frac{4\pi^2}{E(t)}\leq L(t)\leq E(t)\leq E_0.
\end{equation}
For $z=(\mathcal A+iL)/(4\pi)$ and the hyperbolic metric
$h=(\dd x^2+\dd y^2)/y^2$ on $\mathbb H_+$,
\begin{equation}
 \label{eq:modulus-speed-classical}
 |\dot z|_h^2=\frac{(\mathcal A')^2+(L')^2}{L^2}
 \leq\frac{4E\Dcal}{L^2}
 \leq\frac{E_0^3}{4\pi^4}\Dcal.
\end{equation}
For $0\leq s<t<\infty$, the length of the marked-modulus path satisfies
\begin{equation}
 \label{eq:modulus-finite-interval}
 \operatorname{Length}_h(z|_{[s,t]})
 \leq\frac{E_0^{3/2}}{2\pi^2}
       \sqrt{t-s}\sqrt{E(s)-E(t)}.
\end{equation}
Also,
\begin{equation}
 \label{eq:sqrt-length-classical}
 |\sqrt{L(t)}-\sqrt{L(s)}|
 \leq\sqrt{\frac{E_0^2}{4\pi^2}-1}\,
       \sqrt{t-s}\sqrt{E(s)-E(t)}.
\end{equation}
Let $\mathcal A_\infty$ be the representative of the limiting area class
selected by the chosen continuous lift. On a final interval, with
$\Delta(t)=E(t)-E_\infty$ and $\theta$ from Corollary~\ref{cor:LS-tail},
\begin{equation}
 \label{eq:modulus-tail-classical}
 |L(t)-L_\infty|+|\mathcal A(t)-\mathcal A_\infty|
 +\operatorname{Length}_h(z|_{[t,\infty)})
 \leq C\Delta(t)^\theta.
\end{equation}
All hyperbolic path-length estimates hold with exactly the same constants
for the covered modulus $z_M=(c+ez)/(a+bz)$ in
\eqref{eq:modulus-main}, and hence bound distance in the quotient moduli
space.
\end{proposition}

\begin{proof}
Let $f$ be the component of the profile velocity along the unit normal
$n=\gamma\times T$, where $T$ is the oriented unit tangent. The first variations are
\[
 \mathcal A'=-\int f\dd s,\qquad L'=-\int k f\dd s.
\]
Indeed, with velocity $f n+gT$, the oriented swept cylinder in
\eqref{eq:area-lift-local} gives $\mathcal A'=\int\vol_{\Sph^2}(fn+gT,T)\dd s
=-\int f\dd s$. Differentiating the arclength
element gives $\partial_t\dd s=(\partial_r g-kf)\dd s$, and the derivative
of $g$ integrates to zero around the closed profile. 
The classical Willmore velocity $f=-4\Gcal$ therefore gives
\begin{equation}
 \label{eq:area-length-velocities}
 \mathcal A'=4\int\Gcal\dd s,\qquad L'=4\int k\Gcal\dd s.
\end{equation}
Cauchy--Schwarz, with $\int k^2\dd s=E-L$ and
$4\int\Gcal^2\dd s=\Dcal$, proves the first two estimates.
For any $(a_0,b_0)\in\Sph^1\subset\R^2$,
\[
 |a_0\mathcal A'+b_0L'|^2
 \leq4\Dcal\int(a_0+b_0k)^2\dd s
 =4\Dcal\begin{pmatrix}a_0&b_0\end{pmatrix}
 \begin{pmatrix}L&C_1\\C_1&E-L\end{pmatrix}
 \begin{pmatrix}a_0\\b_0\end{pmatrix}.
\]
Taking the supremum over all $(a_0,b_0)\in\Sph^1$ gives the first
inequality in \eqref{eq:joint-classical-bound}. The positive semidefinite
matrix displayed here is the Gram matrix of $1,k$ in $L^2(\dd s)$.
Its trace is $E$ and its larger eigenvalue is
\eqref{eq:moment-eigenvalue}, proving the second inequality.
For any nongeodesic round latitude, $k$ and $\Gcal=k^3+k$ are nonzero
constants, so equality holds in \eqref{eq:sharp-classical-length}.
Such a curve is a smooth admissible initial profile, proving sharpness.

To prove the lower length bound, parametrize any closed profile by
arclength in $\R^3$ and let $T=\partial_r\gamma$. It has mean zero
because the curve is closed. Periodic Wirtinger's inequality yields
\[
 \int_0^L|T'|^2\dd r\geq\frac{4\pi^2}{L^2}
                         \int_0^L|T|^2\dd r=\frac{4\pi^2}{L}.
\]
Since $T'=-\gamma+kn$ on the unit sphere, the left side equals $E$.
The remaining inequalities in \eqref{eq:profile-length-window} follow
from $E=L+\int k^2$ and monotonicity. In particular $E_0\geq2\pi$.
For $z=(\mathcal A+iL)/(4\pi)$ the hyperbolic metric $h$ gives
$|\dot z|_h^2=((\mathcal A')^2+(L')^2)/L^2$. Applying
\eqref{eq:joint-classical-bound} and then \eqref{eq:profile-length-window}
gives \eqref{eq:modulus-speed-classical}. Cauchy--Schwarz in time and
$\int_s^t\Dcal=E(s)-E(t)$ prove
\eqref{eq:modulus-finite-interval}. Likewise,
\[
 |(\sqrt L)'|^2=\frac{(L')^2}{4L}
 \leq\left(\frac EL-1\right)\Dcal
 \leq\left(\frac{E_0^2}{4\pi^2}-1\right)\Dcal,
\]
which proves \eqref{eq:sqrt-length-classical}, including the stationary
case $E_0=2\pi$.

Corollary~\ref{cor:LS-tail} and $\tau=4t$ imply
$\int_t^\infty\sqrt{\Dcal(s)}\dd s\leq C\Delta(t)^\theta$.
The pointwise estimates above and the uniform length and energy bounds
then bound the integrals of $|L'|$, $|\mathcal A'|$, and $|\dot z|_h$
by the same tail. The chosen area lift differs from the local lift in
\eqref{eq:area-lift-local} by a constant in $4\pi\mathbb Z$ on the tail.
Its limit $\mathcal A_\infty$ differs by that same constant, which proves
\eqref{eq:modulus-tail-classical}.

For the fixed cover in \eqref{eq:modulus-main}, let $d=ae-bc>0$ be
its degree. Then $w(z)=(c+ez)/(a+bz)$ satisfies
$w'(z)=d/(a+bz)^2$ and
$\operatorname{Im}w=d\operatorname{Im}z/|a+bz|^2$.
Consequently $|\dot w|_h=|\dot z|_h$ exactly. Passing to the quotient
can only decrease distance.
\end{proof}

\begin{corollary}[Integrable coordinate velocities]
\label{cor:coordinate-velocities}
For every trajectory in Theorem~\ref{thm:main}, the functions
$L',\mathcal A'$ and $|\dot z|_h$ belong to both
$L^1(0,\infty)$ and $L^2(0,\infty)$ and tend to zero. In particular,
\begin{equation}
 \label{eq:coordinate-L2}
 \int_0^\infty\{(L')^2+(\mathcal A')^2\}\dd t
 \leq4E_0(E_0-E_\infty).
\end{equation}
At a covered equator limit their $L^1$ tails and the marked-modulus
distance decay exponentially, and the profile length satisfies \eqref{eq:quadratic-length-main}.
\end{corollary}

\begin{proof}
Integrate \eqref{eq:joint-classical-bound} and
\eqref{eq:modulus-speed-classical} to obtain $L^2$ integrability and
\eqref{eq:coordinate-L2}. By \eqref{eq:joint-classical-bound} and \eqref{eq:modulus-speed-classical},
$|L'|+|\mathcal A'|+|\dot z|_h\leq C(E_0)\sqrt{\Dcal}$.
Integrating this inequality on $[t,\infty)$ gives a bound
$C\Delta(t)^\theta$ by Corollary~\ref{cor:LS-tail}. Smoothness on the
initial compact interval gives the claimed $L^1$ integrability.
Smooth convergence of the constant-speed profiles to a critical curve
implies $\Dcal(t)\to0$. At an equator limit take $\theta=1/2$ and apply
Theorem~\ref{thm:equator-rates}.
\end{proof}

\begin{proposition}[Differential of the marked conformal coordinates]
\label{prop:classical-conformal-rank}
Let $\gamma$ be a smooth regular closed profile and fix a local real lift
of its area class. Write $k$ for its signed geodesic curvature, $L$ for its
length, and $C_1=\int k\,\dd s$. Under a smooth variation whose velocity has normal component $f n_\gamma$,
\begin{equation}\label{eq:classical-conformal-differential}
 D(\mathcal A,L)_\gamma[f]
 =-\left(\int f\,\dd s,\;\int kf\,\dd s\right).
\end{equation}
The marked conformal map on smooth Hopf variations has rank two 
if and only if $k$ is nonconstant. In that case a right inverse of
\eqref{eq:classical-conformal-differential} is
\begin{equation}\label{eq:classical-conformal-right-inverse}
 R_\gamma(u,v):=-\begin{pmatrix}1&k\end{pmatrix}
 \begin{pmatrix}L&C_1\\ C_1&\int k^2\,\dd s\end{pmatrix}^{-1}
 \binom uv,\qquad (u,v)\in\R^2.
\end{equation}
If $k$ is constant, the rank is one. These ranks also hold for every
fixed finite-cover marking.
\end{proposition}

\begin{proof}
The first variations in the proof of Proposition~\ref{prop:classical-estimates}
give \eqref{eq:classical-conformal-differential} for every smooth $f$.
The Gram matrix $G$ of $1$ and $k$ in $L^2(\dd s)$ has determinant
\begin{equation}
 \label{eq:classical-gram-determinant}
 \det G=L\int k^2\,\dd s-C_1^2
 =L\int(k-\bar k)^2\,\dd s,\qquad \bar k=C_1/L.
\end{equation}
Thus $G$ is invertible precisely when $k$ is nonconstant. In this case
substitution shows that \eqref{eq:classical-conformal-right-inverse} is a
right inverse with smooth image. For $y\in\R^2$ it satisfies
$\|R_\gamma y\|_{L^2(\dd s)}^2=y^\top G^{-1}y$. Its operator norm from
Euclidean $\R^2$ to $L^2(\dd s)$ is therefore $\lambda_{\min}(G)^{-1/2}$.
Since $\lambda_{\min}(G)^2\leq\det G$, this norm tends to infinity
as the positive determinant in \eqref{eq:classical-gram-determinant}
tends to zero. If $k$ is constant, the second functional in
\eqref{eq:classical-conformal-differential} is $k$ times the first,
and the first is nonzero since $L>0$. The rank is then one.

To lift a profile variation $\gamma_\epsilon$ with $\gamma_0=\gamma$, use
the pullback model \eqref{eq:pullback-model} used in Proposition~\ref{prop:dictionary}.
For each $x$ on the parameter circle, horizontal transport along
$\epsilon\mapsto\gamma_\epsilon(x)$ gives a smooth bundle isomorphism
$J_\epsilon\colon M_\gamma\to M_{\gamma_\epsilon}$ over the identity,
with $J_0=\id$. For any fixed cover $Q\colon\Sigma\to M_\gamma$, the
immersions $X_{\gamma_\epsilon}\circ J_\epsilon\circ Q$ realize the
variation on $\Sigma$, as in \eqref{eq:constructed-Hopf-family}.
The change from $(\mathcal A,L)$ to $z=(\mathcal A+iL)/(4\pi)$ preserves
rank. For a fixed cover of degree $d=ae-bc>0$, the derivative of
$z\mapsto(c+ez)/(a+bz)$ is $d/(a+bz)^2\ne0$, so its marked modulus has
the same rank.
\end{proof}

Proposition~\ref{prop:nonprojection} shows that two profiles at which the
differential has rank two can have the same marked modulus but
noncollinear projected velocities.

\section{Multiplicity resolution}
\label{sec:multiplicity}

\begin{proposition}[Density of a Hopf immersion]
\label{prop:density}
Let $X_\gamma\colon M_\gamma\to\Sph^3$ be the standard parametrized Hopf
immersion associated with a smooth regular closed curve
$\gamma\colon\Sph^1\to\Sph^2$.  Then, for
$\mathcal H^2$-almost every $z\in X_\gamma(M_\gamma)$,
\begin{equation}
 \label{eq:density-profile}
 \Theta^2(\|\Var_{X_\gamma}\|,z)
 =\#\gamma^{-1}(\pi(z)).
\end{equation}
If $\gamma=\bar\gamma\circ p_m$, where $p_m\colon\Sph^1\to\Sph^1$ is a
smooth covering of degree $m\geq1$ and $\bar\gamma$ is embedded, then
\begin{equation}
 \label{eq:density-mcover}
 \Var_{X_\gamma}=m\,|\pi^{-1}(\bar\gamma(\Sph^1))|.
\end{equation}
\end{proposition}

\begin{proof}
For every smooth proper immersion, the area formula identifies the density of
its induced integral varifold with the cardinality of the fibre of the
immersion at almost every image point.  Here
$\Theta^2(\mu,z):=\lim_{r\downarrow0}\mu(B_r(z))/(\pi r^2)$ denotes the
two-dimensional density, which exists at almost every point of the
rectifiable measure in question.  The fibre identity \eqref{eq:standard-model-fibre} gives
$\#X_\gamma^{-1}(z)=\#\gamma^{-1}(\pi(z))$ for every image point.
Combining these facts proves \eqref{eq:density-profile}. This
argument also covers profile arcs which are retraced and therefore does not
require transverse self-intersections.  For an $m$-fold cover of an embedded
curve, the right-hand side equals $m$ almost everywhere on the underlying
Hopf torus; the equality of integral varifolds follows.
\end{proof}

\begin{proof}[Completion of the proofs of
Theorems~\ref{thm:main} and \ref{thm:endpoint}]
Proposition~\ref{prop:parametrized-reduction} proves the assertions of 
Theorem~\ref{thm:main}(i), including global existence, the parametrized normalization, and the finite-time multiplicity formula.  Proposition~\ref{prop:profile-full} gives the full
limiting profile, and Proposition~\ref{prop:smooth-lift} lifts it to the
surface convergence in Theorem~\ref{thm:main}\textnormal{(ii)}.
Corollary~\ref{cor:LS-tail} and Proposition~\ref{prop:surface-length} give
Theorem~\ref{thm:endpoint}\textnormal{(i)} and
\eqref{eq:energy-rates-main}.  Proposition~\ref{prop:transport}, followed by
Proposition~\ref{prop:smooth-lift}, gives
Theorem~\ref{thm:endpoint}\textnormal{(ii)}.  Lemma~\ref{lem:anchored-lift}
and Proposition~\ref{prop:smooth-lift} give
Theorem~\ref{thm:endpoint}\textnormal{(iv)}, while
Proposition~\ref{prop:density}, with the area formula for the
$d$-sheeted cover $q_\infty$, proves
Theorem~\ref{thm:endpoint}\textnormal{(iii)}.
\end{proof}

\section{The multiplicity-two Clifford endpoint}
\label{sec:double-Clifford}

The existence of simple profiles converging to a twice-traversed great
circle was proved in \cite[Theorem~2.2]{JakobMultiplicity}. We give an
explicit family and obtain exponential convergence from the attracting
neighborhood proved in Section~\ref{sec:equators}.

\begin{corollary}[Double-Clifford endpoint above $4\pi^2$]
\label{cor:double-Clifford}
For every $\delta>0$ there is a smooth simple initial Hopf immersion $F_0$
in the sense of \eqref{eq:ae-simple} such that
\begin{equation}
 \label{eq:sharp-initial-energy}
 4\pi^2<\Will(F_0)<4\pi^2+\delta
\end{equation}
and its classical Willmore flow has the following full endpoint:
\begin{align}
 \Will(F_t)&\downarrow4\pi^2,
 \label{eq:double-energy}\\
 F_t\circ\Psi_t&\longrightarrow F_\infty
 \quad\text{exponentially in every }C^k,
 \label{eq:double-smooth}\\
 \Var_t&\rightharpoonup2|T_{\mathrm{Cl}}|,
 \qquad
 (F_t)_\#\mu_t\rightharpoonup
 2\mathcal H^2\llcorner T_{\mathrm{Cl}},
 \label{eq:double-varifold}
\end{align}
where $T_{\mathrm{Cl}}$ is a Clifford torus and $F_\infty$ covers it exactly
twice. The limiting profile has length $L_\infty=4\pi$ and area class
$[0]\in\R/(4\pi\mathbb Z)$. With the real representative $\mathcal A_\infty=0$,
\begin{equation}
 \label{eq:double-lattice}
 \Gamma_\infty
 =\Span_{\mathbb Z}\{(2\pi,0),(0,2\pi)\},
 \qquad
 \left[\frac{\mathcal A_\infty+iL_\infty}{4\pi}\right]=[i].
\end{equation}
\end{corollary}

\begin{proof}
On $\R/(4\pi\mathbb Z)$ take the graph
\[
 \gamma_\epsilon(s)=
 \bigl(\cos(\epsilon\sin(s/2))\cos s,
       \cos(\epsilon\sin(s/2))\sin s,
       \sin(\epsilon\sin(s/2))\bigr),\qquad 0<|\epsilon|\ll1.
\]
It is regular because $\cos(\epsilon\sin(s/2))>0$. Equality of two
image points forces equal longitude modulo $2\pi$. The only distinct
parameters with equal image are $s=0$ and $s=2\pi$: shifting $s$ by
$2\pi$ changes the sign of the height, so equality requires
$\sin(s/2)=0$. The tangents at this double point are
$(0,1,\epsilon/2)$ and $(0,1,-\epsilon/2)$ and are independent.
Thus the profile, and its standard Hopf immersion $F_\epsilon$, are simple in the
a.e. sense. The Fourier mode $\sin(s/2)$ has Hessian eigenvalue $3/8$.
Equation~\eqref{eq:graph-energy-expansion} gives
\[
 \El(\gamma_\epsilon)=4\pi+\frac{3\pi}{8}\epsilon^2+O(\epsilon^4),
 \qquad
 \Will(F_\epsilon)=4\pi^2+\frac{3\pi^2}{8}\epsilon^2+O(\epsilon^4).
\]
The remainder has even order because reflection across the equator
interchanges $\epsilon$ and $-\epsilon$ and preserves the energy.
Since $\|\epsilon\sin(s/2)\|_{H^2}\to0$, choose $|\epsilon|$ small enough
for Theorem~\ref{thm:equator-stability} with $m=2$ and so that
\eqref{eq:sharp-initial-energy} holds for $F_0=F_\epsilon$. That theorem gives exponential
convergence to a rotated twice-covered equator. Its length is $4\pi$,
its area class is $4\pi=0$ modulo $4\pi$, and its energy is $4\pi$.

Theorem~\ref{thm:equator-rates} gives the surface rate in
\eqref{eq:double-smooth}, and Theorem~\ref{thm:endpoint} gives
\eqref{eq:double-varifold}.  Substitution of
$\mathcal A_\infty=0$ and $L_\infty=4\pi$ into Pinkall's formula yields
\eqref{eq:double-lattice}.
\end{proof}

\begin{remark}[Why the modulus does not resolve multiplicity]
The standard Hopf model over a twice-traversed equator has the lattice in
\eqref{eq:double-lattice}, of covolume $4\pi^2$. The standard model over a
once-traversed equator has covolume $2\pi^2$ and marked modulus
$z_0=(1+i)/2$ for the area lift $2\pi$. The modular map $z\mapsto1-1/z$
sends $z_0$ to $i$. Thus the modular point $[i]$ does not
distinguish these two immersions. Their marked lattices, limiting masses,
and the density formula \eqref{eq:multiplicity-main} do distinguish them.
A degree-two Clifford cover can also have a nonsquare conformal class.
For this once-traversed equator, the cover determined by the matrix $M=\operatorname{diag}(2,1)$ has marked modulus $z_M=(1+i)/4$
in formula~\eqref{eq:modulus-main}. The modular transformation
$z\mapsto-1/z+2$ sends this point to $2i$. Every modular image of $i$
has imaginary part $1/(r^2+s^2)\leq1$, where $(r,s)$ is the bottom row
of a matrix in $\mathrm{SL}_2(\mathbb Z)$. Hence $[2i]\ne[i]$.
\end{remark}

\section{Failure of autonomous conformal evolution}
\label{sec:nonprojection}

For the standard marking, Pinkall's formula identifies the conformal
parameter with $z=(\mathcal A+iL)/(4\pi)$. We construct two profiles
with exactly the same area and length and compute their initial Willmore velocities in these coordinates. A positive change of time is allowed to  rescale each velocity separately.

\begin{proposition}[Noncollinear velocities at one conformal class]
\label{prop:nonprojection}
For every sufficiently small $\delta>0$ there are two smooth embedded
profiles $\gamma_{1,\delta},\gamma_{2,\delta}$ with
$\mathcal A=2\pi$ and $L=2\pi+\delta$ for which the instantaneous
classical Willmore velocities in the $(\mathcal A,L)$ 
plane are not collinear. Their standard degree-one Hopf tori have Willmore energy less than $8\pi$.
At both tori the map to marked conformal class has rank two on smooth 
Hopf variations. Both profiles tend smoothly to the once-covered equator as $\delta\downarrow0$.
Consequently, no vector field near its marked conformal class can describe
the projections of all classical Willmore-Hopf flows whose initial profiles lie in a $C^\infty$ neighborhood of that equator, even after multiplication of their velocities by arbitrary positive scalar factors. The same obstruction holds in unmarked moduli space at these common conformal
classes. For each fixed finite-cover marking it persists in Teichm\"uller
space and, after possibly shrinking the interval of positive $\delta$,
in moduli space. This finite-cover assertion requires no Willmore energy
threshold.
\end{proposition}

\begin{proof}
Work on $\R/(2\pi\mathbb Z)$ in the graph notation
\eqref{eq:equator-graph}. For graphs with $\cos u>0$, longitude is the
parameter itself, so the curve is embedded. The north-cap area and the
Gauss--Bonnet identity in this chart are
\begin{equation}
 \label{eq:graph-area-GaussBonnet}
 \mathcal A(u)=2\pi-\int_0^{2\pi}\sin u\dd s,
 \qquad \int k_u\dd s_{\Gamma_u}=\int_0^{2\pi}\sin u\dd s.
\end{equation}
The area formula follows by integrating the spherical area form between
latitude $u(s)$ and the north pole; the second identity is Gauss--Bonnet
for this embedded oriented boundary.

Write $f_1(s)=\cos 2s$ and $f_2(s)=\cos 2s+\cos 4s$. For small real
$\varepsilon$, take the two graph profiles $\Gamma_{u^j_\varepsilon}$
of \eqref{eq:equator-graph}, where
\[
 u^1_\varepsilon(s)=\varepsilon f_1(s),\qquad
 u^2_\varepsilon(s)=c(\varepsilon)+\varepsilon f_2(s).
\]
The first graph has area $2\pi$ since translation by $\pi/2$ changes the
sign of $u^1_\varepsilon$. For the second graph we choose $c(0)=0$ and
require $\int_0^{2\pi}\sin(c+\varepsilon f_2)\dd s=0$. The derivative
of this integral with respect to $c$ at $(c,\varepsilon)=(0,0)$ is $2\pi$.
The analytic implicit function theorem gives a unique small analytic
$c(\varepsilon)$, and uniqueness under $(c,\varepsilon)\mapsto(-c,-\varepsilon)$
makes it odd. Since $\int f_2\dd s=0$ and $\int f_2^3\dd s=3\pi/2$,
expansion of the sine yields
\begin{equation}
 \label{eq:area-correction}
 c(\varepsilon)=\varepsilon^3/8+O(\varepsilon^5).
\end{equation}
Both graph families thus have area exactly $2\pi$. Their arclength
densities $v_j=(\cos^2 u^j_\varepsilon+(\partial_s u^j_\varepsilon)^2)^{1/2}$
are even and analytic in $\varepsilon$. Integrating their expansions gives
\begin{align}
 L(\Gamma_{u^1_\varepsilon})-2\pi&=\tfrac{3\pi}{2}\varepsilon^2+O(\varepsilon^4),
 \label{eq:pure-length-expansion}\\
 L(\Gamma_{u^2_\varepsilon})-2\pi&=9\pi\varepsilon^2+O(\varepsilon^4).
 \label{eq:mixed-length-expansion}
\end{align}
As functions of $\varepsilon^2$, both length excesses have positive
derivative at zero. For every sufficiently small $\delta>0$ they
therefore determine unique small positive amplitudes $\varepsilon_j(\delta)$
for which the two profiles
$\gamma_{j,\delta}:=\Gamma_{u^j_{\varepsilon_j(\delta)}}$ have length
exactly $2\pi+\delta$.

Write $k_j(\varepsilon,s):=k_{u^j_\varepsilon}(s)$ for the signed
geodesic curvature in the fixed longitude coordinate $s$. For the graphs \eqref{eq:equator-graph}, both heights are odd and
analytic in $\varepsilon$. In the curvature formula \eqref{eq:graph-curvature}
the numerator is therefore odd and the denominator $v_j^3$ is even
and positive. It follows that $k_j$ is odd and analytic in the amplitude.
Since $\partial_\varepsilon u^j_\varepsilon|_{\varepsilon=0}=f_j$ by \eqref{eq:area-correction},
differentiation of \eqref{eq:graph-curvature} gives the first Taylor
coefficient $q_j(s):=\partial_\varepsilon k_j(0,s)=f_j''(s)+f_j(s)$.
In particular,
\[
 \begin{aligned}
 k_j(\varepsilon,s)&=\varepsilon q_j(s)+O(\varepsilon^3),\\
 q_1(s)&=-3\cos 2s,\qquad q_2(s)=-3\cos 2s-15\cos 4s .
 \end{aligned}
\]
The remainder is uniform with any fixed number of $s$-derivatives, by
the same formula on the compact parameter circle.

Let primes on $\mathcal A$ and $L$ denote derivatives in Willmore time
at the initial profile, and let $r$ denote its spherical arclength.
Integration by parts in \eqref{eq:area-length-velocities} gives
\begin{equation}
 \label{eq:integrated-classical-velocities}
 \mathcal A'=4\int(k^3+k)\dd s_\gamma,
 \qquad L'=4\int(k^4+k^2-2k_r^2)\dd s_\gamma.
\end{equation}
The integral of $k$ vanishes in both families by
\eqref{eq:graph-area-GaussBonnet}. In the first family translation by
$\pi/2$ reverses $k_1$ and preserves $v_1$, so $\mathcal A'_1=0$
exactly. Using $\partial_r=v_1^{-1}\partial_s$ and $v_1=1+O(\varepsilon^2)$,
the quadratic term of $L'_1$ is
$4\varepsilon^2\int(q_1^2-2(q_1')^2)\dd s=-252\pi\varepsilon^2$.
The remaining terms are $O(\varepsilon^4)$ by parity. At the amplitude
determined by \eqref{eq:pure-length-expansion} this gives
\begin{equation}
 \label{eq:pure-projected-velocity}
 (\mathcal A'_1,L'_1)=(0,-168\delta+O(\delta^2)).
\end{equation}
For the second family, Fourier orthogonality gives
\[
 \int q_2^3\dd s=-\frac{405\pi}{2},\qquad
 \int q_2^2\dd s=234\pi,\qquad
 \int (q_2')^2\dd s=3636\pi.
\]
In the cubic integral only the term
$3(-3)^2(-15)\cos^2(2s)\cos 4s$ has nonzero mean, and
$\int_0^{2\pi}\cos^2(2s)\cos 4s\dd s=\pi/2$.
Substitution into \eqref{eq:integrated-classical-velocities}, with
$\dd s_\gamma=v_2\dd s$ and $\partial_r=v_2^{-1}\partial_s$, gives
\[
 \mathcal A'_2=-810\pi\varepsilon^3+O(\varepsilon^5),\qquad
 L'_2=-28152\pi\varepsilon^2+O(\varepsilon^4).
\]
At the amplitude determined by \eqref{eq:mixed-length-expansion} this gives
\begin{equation}
 \label{eq:mixed-projected-velocity}
 (\mathcal A'_2,L'_2)=
 \left(-\frac{30}{\sqrt\pi}\delta^{3/2}+O(\delta^{5/2}),
       -3128\delta+O(\delta^2)\right).
\end{equation}
The determinant of the columns in
\eqref{eq:pure-projected-velocity} and
\eqref{eq:mixed-projected-velocity} is
$-5040\delta^{5/2}/\sqrt\pi+O(\delta^{7/2})$, which is nonzero for
all sufficiently small positive $\delta$.

As $\delta\downarrow0$, both amplitudes tend to zero and the profiles
converge smoothly to the once-traversed equator. By \eqref{eq:W=piE},
the Willmore energies of their associated simple Hopf tori tend to
$2\pi^2<8\pi$. For both families the first Taylor coefficient satisfies
$q_j(0)-q_j(\pi/2)=-6$. Thus
$k_j(\varepsilon,0)-k_j(\varepsilon,\pi/2)=-6\varepsilon+O(\varepsilon^3)$,
which is nonzero at each small positive $\varepsilon_j(\delta)$.
The curvatures are therefore nonconstant, and
Proposition~\ref{prop:classical-conformal-rank} gives rank two at both tori.

The common marked point is $z_\delta=1/2+i(1/2+\delta/(4\pi))$.
As in \eqref{eq:double-lattice}, the map $(\mathcal A,L)\mapsto z\in\mathbb H_+$
identifies $(\mathcal A,L)$ with $(\mathcal A+iL)/(4\pi)$ and hence
preserves noncollinearity. If the two
projected flows followed a single vector field up to positive changes
of time, their velocities at $z_\delta$ would both be scalar multiples of
its value there. This contradicts the determinant calculation above.
Since $\gamma_{j,\delta}\to\eta_1$ smoothly and $z_\delta\to z_0=(1+i)/2$,
the construction applies in every neighborhood specified in the statement.

For any fixed cover marking $M$, including the standard marking $M=I$,
write $w_M(z)=(c+ez)/(a+bz)$ as in \eqref{eq:modulus-main}. Its
derivative $d/(a+bz)^2$ is nonzero, so it preserves rank and
noncollinearity at $z_\delta$. The path $w_M(z_\delta)$ is injective and
tends to $w_M(z_0)$. The elliptic points of the modular group are
discrete in $\mathbb H_+$, so some neighborhood of $w_M(z_0)$ contains
no elliptic point except possibly $w_M(z_0)$ itself. The modular
quotient is therefore a local diffeomorphism at $w_M(z_\delta)$ for all
sufficiently small positive $\delta$, which proves the unmarked assertion.
The $8\pi$ bound applies only to the two standard degree-one tori.
By \eqref{eq:W=piE}, a degree-$d$ cover of either torus has $d$ times
its Willmore energy.
\end{proof}

\section{Comparison with the M\"obius-invariant flow}
\label{sec:comparison-flows}

For both profile equations we measure velocity in the unit round sphere,
taking its signed component along $n=\gamma\times T$.

\begin{proposition}[The two profile vector fields]
\label{prop:flow-comparison}
Let $\gamma$ be a smooth regular profile with signed geodesic curvature $k$
and set $G=2\partial_r^2 k+k^3+k$, where $r$ is spherical arclength. With the curvature convention of \cite[Proposition~4.2,
(77)]{JakobMIWF} and the MIWF reduction in
\cite[Proposition~4.4, (91)]{JakobMIWF}, the scalar normal velocities of the classical and M\"obius-invariant Willmore flows, respectively, are
\begin{equation}\label{eq:two-profile-velocities}
 V_{\mathrm{cl}}=-4G,\qquad
 V_{\mathrm{MI}}=-\frac{G}{(1+k^2)^2}.
\end{equation}
At a nonstationary profile, there exists a single $a>0$ such that
$V_{\mathrm{MI}}=aV_{\mathrm{cl}}$ along the whole curve if and only if
$(1+k^2)^{-2}$ is constant on $\{G\ne0\}$. At a stationary profile both
velocities vanish. There are smooth
embedded profiles arbitrarily close to the equator for which no such scalar
exists. Consequently, for those initial profiles the two geometric profile
flows cannot agree on any interval $[0,T]$, $T>0$, under a $C^1$ time
change fixing zero with positive derivative there, even allowing
time-dependent reparametrizations of the curve.

At every $m$-fold equator the linearizations of the graph equations are
respectively $-4H_m$ and $-H_m$, with $H_m$ as in
\eqref{eq:equator-Hessian}; in particular their linearized stable, unstable,
and kernel subspaces coincide.
\end{proposition}

\begin{proof}
Equations \eqref{eq:profile-t} and \eqref{eq:elastic-gradient-explicit}
give $V_{\mathrm{cl}}=-4G$. The M\"obius-invariant flow multiplies the
classical Willmore surface velocity by $|A^\circ|^{-4}$. 
By formula \eqref{eq:Hopf-second-fundamental} 
we have $|A^\circ|^2=2(1+k^2)$, and this yields the second identity in \eqref{eq:two-profile-velocities} on every fixed finite cover.

If $G\equiv0$, both profile equations have zero perpendicular velocity. Otherwise, wherever $G\ne0$, division in
\eqref{eq:two-profile-velocities} gives
\[
 \frac{V_{\mathrm{MI}}}{V_{\mathrm{cl}}}=\frac{1}{4(1+k^2)^2}.
\]
At every zero of $G$ both velocities vanish. Thus a single time-rescaling
factor for the whole profile exists exactly when this ratio is constant
on $\{G\ne0\}$.

Consider the embedded graph $u_\varepsilon(s)=\varepsilon\cos 2s$ over
the once-traversed equator, with longitude $s$. By
\eqref{eq:graph-curvature} and $\partial_r=v^{-1}\partial_s$, where
$v=1+O(\varepsilon^2)$ is the graph speed,
\[
 k_\varepsilon=-3\varepsilon\cos 2s+O(\varepsilon^3),\qquad
 G_\varepsilon=21\varepsilon\cos 2s+O(\varepsilon^3).
\]
The curvature expansion holds with its $s$-derivatives, as in the proof
of Proposition~\ref{prop:nonprojection}. At $s=0$ and $s=\pi/6$ the
first Taylor coefficients of $G_\varepsilon$ are $21$ and $21/2$, so
the velocity ratio is defined at both points for small $\varepsilon\ne0$.
At those same points,
\[
 k_\varepsilon(0)^2-k_\varepsilon(\pi/6)^2
 =\tfrac{27}{4}\varepsilon^2+O(\varepsilon^4)\ne0.
\]
Since $x\mapsto1/[4(1+x)^2]$ is strictly decreasing on $[0,\infty)$,
the two ratios are different. To see the consequence for the flows,
suppose that $\gamma_{\mathrm{MI}}(t)
=\gamma_{\mathrm{cl}}(\sigma(t))\circ\phi_t$ with $\sigma(0)=0$ and
$\sigma'(0)>0$. The common initial parametrization is embedded, so
$\phi_0=\id$. Differentiation at $t=0$ and projection onto the profile
normal remove the tangential term from $\dot\phi_0$ and give
$V_{\mathrm{MI}}(s)=\sigma'(0)V_{\mathrm{cl}}(s)$ at every $s$.
The distinct ratios at $0$ and $\pi/6$ contradict this identity.

For a graph $\Gamma_u$ as in \eqref{eq:equator-graph} with speed
$v=|\partial_s\Gamma_u|$, the oriented unit normal is
$N_u=\Gamma_u\times\partial_s\Gamma_u/v$, so
$\langle\partial_u\Gamma_u,N_u\rangle=\cos u/v$. At fixed longitude the
scalar normal velocity $V=\langle\partial_t\Gamma_{u(t)},N_u\rangle$
and the time derivative of the graph height are related by
\begin{equation}
 \label{eq:fixed-graph-velocity}
 \partial_tu=(v/\cos u)V.
\end{equation}
At the $m$-fold equator $u=0$, the graph speed is $v=1$ and the
conversion factor $v/\cos u$ equals one. Also $k=G=0$ there.
Linearizing \eqref{eq:graph-curvature} gives $Dk(0)[h]=h_{ss}+h$,
and $\partial_r=\partial_s$ there. In differentiating $\partial_r^2 k$,
the variation of the arclength operator acts on the identically zero
curvature of the equator. Hence
\[
 DG(0)[h]=(2\partial_s^2+1)(\partial_s^2+1)h=H_mh
\]
by \eqref{eq:equator-Hessian}. In the two graph evolution equations obtained from
\eqref{eq:fixed-graph-velocity} and \eqref{eq:two-profile-velocities},
every derivative of a factor multiplying $G$ is multiplied by $G(0)=0$.
The remaining terms are $-4H_mh$ and $-H_mh$, respectively.
\end{proof}

\section{Low-energy specialization and scope}
\label{sec:low-energy}

The following consequence recovers \cite[Theorem~1.1]{JakobMultiplicity}.

\begin{corollary}[Sub-$4\pi^2$ Clifford convergence]
\label{cor:low-energy}
If, in addition to the hypotheses of Theorem~\ref{thm:main},
$\Will(F_0)<4\pi^2$, then $F_\infty$ is a diffeomorphism onto a
Clifford torus $T_{\mathrm{Cl}}$, up to an isoclinic rotation of $\Sph^3$.
Moreover, $F_t$ is an embedding for all sufficiently large $t$, and
$\Var_\infty=|T_{\mathrm{Cl}}|$.
\end{corollary}

\begin{proof}
By \eqref{eq:profile-length-window}, every profile satisfies
$\El(\gamma)^2\geq L(\gamma)\El(\gamma)\geq4\pi^2$. Thus
$\Will(F_0)=d\pi\El(\gamma_0)<4\pi^2$ forces $d=1$.
The critical value $\El(\gamma_\infty)\in[2\pi,4\pi)$ equals $2\pi$
by Jakob's classification \cite[p.~3, discussion preceding Theorem~1.1]{JakobMultiplicity}.
Equality in both length inequalities gives $L_\infty=2\pi$ and
$\int k_\infty^2\dd s=0$. Hence $\gamma_\infty$ traverses a great
circle $C$ once. Proposition~\ref{prop:smooth-lift} gives
$F_\infty=X_{\gamma_\infty}\circ q_\infty$, where the degree-one covering
$q_\infty\colon\Sigma\to M_{\gamma_\infty}$ is a diffeomorphism. Thus,
up to reparametrization of $\Sigma$, the limiting immersion is the
standard Hopf parametrization of $\pi^{-1}(C)$. Since $\gamma_\infty$
traverses $C$ once, the projection $(b,q)\mapsto q$ is a diffeomorphism
onto $\pi^{-1}(C)$. Embeddings of the compact torus are open in $C^1$,
so \eqref{eq:smooth-main} makes $F_t\circ\Psi_t$, and hence $F_t$,
an embedding for all sufficiently large $t$.
Any other smooth immersed limit after reparametrization induces the same
varifold by \eqref{eq:varifold-main}. Its image is therefore $\pi^{-1}(C)$,
and it covers this torus with degree one, the varifold multiplicity. Thus
the limiting immersion is unique up to reparametrization.
The inverse image $\pi^{-1}(C)$ is a Clifford torus, and right multiplication by a unit quaternion lifts any
rotation of the base sphere to an isoclinic rotation of $\Sph^3$.
\end{proof}

\section*{Acknowledgment}
The authors acknowledge the use of ChatGPT (OpenAI) for assistance with
mathematical discussions and the preparation of the manuscript. The authors
are fully responsible for the mathematical content and the final text.

\end{document}